\documentclass[a4paper,	11pt]{article}

\usepackage[english]{babel}
\usepackage[T1]{fontenc} 
\usepackage[babel=true]{csquotes} 
\usepackage{aeguill} 

\usepackage{latexsym}
\usepackage{comment}

\usepackage{svg}
\usepackage{hyperref}
\usepackage{tabularx}
\usepackage{amssymb}
\usepackage{amsmath}
\usepackage{amsthm} 
\usepackage{graphicx}
\usepackage{empheq}
\usepackage{color}
\usepackage{empheq}
\usepackage{float}
\usepackage{caption}
\usepackage{subcaption}
\usepackage{xcolor}
\usepackage{adjustbox}
\usepackage{listings}
\usepackage{algorithm}
\usepackage{algpseudocode}
\usepackage[algo2e,ruled,vlined]{algorithm2e}

\definecolor{dkgreen}{rgb}{0,0.6,0}
\definecolor{gray}{rgb}{0.5,0.5,0.5}
\definecolor{mauve}{rgb}{0.58,0,0.82}

\def \d{\mathrm{d}} 
 
\newcommand{\R}{{\bf R}}

\usepackage{amsmath, amssymb}
\usepackage{graphicx}
\usepackage{color}
\usepackage{textcomp}
\usepackage{graphics}
\usepackage{float}
\usepackage{empheq}
\usepackage{caption}
\usepackage{subcaption}

\usepackage{geometry}
\usepackage{stmaryrd}

\AtBeginDocument{}
\AtBeginDocument{}

\newtheorem{thm}{Theorem}[section]

\newtheorem{lem}[thm]{Lemma}
\newtheorem{prop}[thm]{Proposition}

\newtheorem{rem}[thm]{Remark}

\mathtoolsset{showonlyrefs}

\newcommand{\be}{\begin{equation}}
\newcommand{\ee}{\end{equation}}

\newcommand{\bea}{\begin{eqnarray}}
\newcommand{\bes}{\begin{subEquations}}
\newcommand{\ees}{\end{subEquations}}
\newcommand{\bgt}{\begin{gather}}
\newcommand{\egt}{\begin{gather}}
\newcommand{\eea}{\end{eqnarray}}
\newcommand{\beaa}{\begin{eqnarray*}}
\newcommand{\eeaa}{\end{eqnarray*}}

\def \R{\mathbb{R}}

\def \E{\mathbb{E}}
\def \F{\mathbb{F}}

\def \P{\mathbb{P}}

\def \W{\mathbb{W}}

\def \Ac{{\cal A}}

\def \Fc{{\cal F}}

\def \Hc{{\cal H}}

\def \Pc{{\cal P}}

\def \Nc{{\cal N}}

\def \Tc{{\cal T}}

\def\Dt#1{\Frac{\partial #1}{\partial t}}
\def \Frac{\displaystyle\frac}
\def \Sum{\displaystyle\sum}
\def \Prod{\displaystyle\prod}

\def\red#1{{\color{red}#1}}

\def \mra{\mathrm{a}}

\def \bolx{\boldsymbol{x}}
\def \boX{\boldsymbol{X}}
\def \boY{\boldsymbol{Y}}

\numberwithin{equation}{section} 

\begin{document}

\title{Generative modeling for time series via Schr\"odinger bridge }
\author{Mohamed HAMDOUCHE\footnote{LPSM,  Universit\'e Paris Cité and Sorbonne Universit\'e, and Qube Research and Technologies, \sf\href{mailto:hamdouche at lpsm.paris}{hamdouche at lpsm.paris}} \quad  Pierre HENRY-LABORDERE\footnote{Qube Research and Technologies, \sf \href{mailto:pierre.henrylabordere at qube-rt.com}{pierre.henrylabordere at qube-rt.com}} \quad  Huy\^en PHAM\footnote{LPSM,  Universit\'e Paris Cité and Sorbonne Universit\'e,   \sf \href{mailto:pham at lpsm.paris}{pham at lpsm.paris}; This author  is supported by  the BNP-PAR Chair ``Futures of Quantitative Finance", 
and by FiME, Laboratoire de Finance des March\'es de l'Energie, and the ``Finance and Sustainable Development'' EDF - CACIB Chair}}

\date{}

\maketitle

\begin{abstract}  
We propose a novel generative model for time series based on Schr\"odinger bridge (SB) approach. This consists in the entropic interpolation via optimal transport between a reference probability measure on path space and a target measure consistent with the joint data distribution of the time series. The solution is characterized by a stochastic differential equation on finite horizon with a path-dependent drift function, hence respecting  the temporal dynamics of the time series distribution.
We can  estimate the drift function from data samples either by kernel regression methods 
or with LSTM neural networks, and the simulation of the SB diffusion  yields new synthetic data samples of the time series. 

The performance of our generative model is evaluated through a series of numerical experiments.  First, 
we test with a toy autoregressive model, a GARCH Model, and the example of fractional Brownian motion,  and measure the accuracy of our algorithm with marginal and temporal dependencies metrics. Next, we use our SB generated synthetic samples for the application to deep hedging on real-data sets. Finally, we illustrate the SB approach  for generating sequence of images. 
\end{abstract}

\vspace{5mm}


\section{Introduction}

Sequential data appear widely in our society like in video and audio data, and simulation of time series models are used in various industrial applications including clinical predictions \cite{lyuetal18}, and weather forecasts \cite{chenetal18}. In the financial industry, simulations of dynamical scenarios are considered in market stress tests, risk measurement, and risk management, e.g. in deep hedging \cite{buedeep}. The design of time series model is a delicate issue, requiring expensive calibration task, and subject to error misspecification and model risk. Therefore, the generation for synthetic samples of time series has gained an increasing attention over the last years, and opens the door in the financial sector for a pure data driven approach in risk management.

Generative modeling (GM) has become over the last  years a successful machine lear\-ning task for data synthesis notably in (static) image processing. Several competing methods have been developed and state-of-the-art includes 
{\it Likelihood-based models} like energy-based models (EBM) \cite{LeCunetal06}, variational auto-encoders (VAE) \cite{VAE}, {\it Implicit generative models} with the prominent works on generative adversarial network (GAN) \cite{GAN} and its extensions \cite{arjetal17}, and recently the new {\it generation of score-based models} using Langevin dynamics, \cite{sonerm19}, \cite{sonetal21}, \cite{costemallat22},  and diffusion models via Schr\"odinger bridge, see \cite{PHL19} for the application to a class of stochastic volatility models, and 
 \cite{wangetal21}, \cite{debor21}. 
Generative methods for time series raises challenging issues for learning efficiently the temporal dependencies. Indeed, in order to capture the potentially complex dynamics of variables across time, it is not sufficient to learn the time marginals or even the joint distribution without exploiting the sequential structure. An increasing attention has been paid to these methods in the literature and state-of-the-art generative methods for time series are:  
{\it Time series GAN} \cite{yoon19} which  combines an unsupervised adversarial loss on real/synthetic data and supervised loss for generating sequential data,  {\it Quant GAN} \cite{wie20} with an adversarial generator using temporal convolutional networks,  {\it Causal optimal transport COT-GAN}  \cite{xuetal20} with adversarial generator using the adapted Wasserstein distance for processes, {\it Conditional loss Euler generator} \cite{remetal21} starting from a diffusion representation  time series and minimizing the conditional distance between transition probabilities of real/synthetic samples, {\it Signature embedding of time series} \cite{fer19}, \cite{nietal20}, \cite{bueetal20}, and 
{\it Functional data analysis}  with neural SDEs \cite{jaietal23}.

In this paper, we develop a novel generative model based on Schr\"odinger bridge approach that captures the temporal dynamics of the time series.  This consists in the entropic interpolation via optimal transport between a reference probability measure on path space and a target measure consistent with the joint data distribution of the time series. The solution is characterized by a stochastic differential equation on finite horizon with a path-dependent drift function, called Schr\"odinger bridge time series (SBTS)  diffusion, and   
the simulation of the SBTS diffusion yields new synthetic data samples of the time series.  

Our SB approach differs from related works that have been recently designed for lear\-ning  marginal (or static)  distributions. 
In \cite{wangetal21}, the authors perform generative models by solving two SB problems. The paper \cite{debor21} formulates generative modeling by computing the SB problem between the data and prior distribution. The very recent work \cite{deepmomentum} proposes momentum SB by considering an additional velocity variable for learning multi marginal distributions. Let us mention also the recent paper \cite{alignedSB} that combines SB with $h$-transform in order to respect aligned data. 
Instead, our SBTS diffusion interpolates the joint time series distribution starting from an initial deterministic value. Moreover, we propose an alternate method for the estimation of the drift function, which is path-dependent in our case. While \cite{wangetal21} uses a logistic regression for estimating the density ratio and then the drift function, which requires additional samples from Gaussian noises, and \cite{debor21} performs an extension of the Sinkhorn algorithm, we propose a kernel regression method relying solely on data samples, and this turns out to be quite simple, efficient and low-cost computationally. Compared to GAN type methods, the simulation of synthetic samples from SBTS is much faster as it does not require the training of neural networks. 

We validate our methodology with several numerical experiments. We first test on some time series models like autoregressive, GARCH models, and also for the fractional Brownian motion with rough paths. The accuracy is measured by some metrics aiming to capture the temporal dynamics and the correlation structure. We also provide operational metrics of interest for the financial industry by implementing our results on real data-sets, and applying to the deep hedging of call options. Finally, we show some numerical illustrations of our SB method in high dimension for the generation of sequential images. 

\section{Problem formulation}

Let $\mu$ be the distribution of a time series representing the evolution of some  $\R^d$-valued process of interest (e.g. asset price, claim process, audio/video data, etc), and suppose that one can observe samples of this process at given fixed  
times of a discrete time grid $\Tc$ $=$ $\{t_i, i = 1,\ldots,N\}$ on $(0,\infty)$. We set $T$ $=$ $t_N$ as the terminal  observation horizon.  Our goal is to construct a model that generates time series samples according to the unknown target distribution $\mu$ $\in$ $\Pc((\R^d)^N)$ the set of probability measures on $(\R^d)^N$. 
For that objective, we propose a dynamic modification of the  Schr\"odinger bridge as follows.  Let   $\Omega$ $=$ $C([0,T];\R^d)$ be the space of $\R^d$-valued continuous functions on $[0,T]$,  $X$ $=$ $(X_t)_t$ the canonical process  with initial value $X_0$ $=$ $0$, 
and $\F$ $=$ $(\Fc_t)_t$ the canonical filtration. Denoting by $\Pc(\Omega)$ as the space of probability measures on $\Omega$, we search for $\P$ (representing the theoretical generative model) in $\Pc(\Omega)$, 
close to the Wiener measure $\W$   in the sense of Kullback-Leibler (or relative entropy), and consistent with  the observation samples. In other words, we look for a probability measure  $\P^*$ $\in$ $\Pc(\Omega)$ solution to: 
\begin{align} \label{SBP} 
\P^* & \in \; {\rm arg}\min_{\P\in\Pc_{\Tc}^\mu(\Omega)} H(\P|\W),  
\end{align}
where $\Pc_{\Tc}^\mu(\Omega)$ is the set of probability measures $\P$ on $\Omega$ with joint distribution $\mu$ at $(t_1,\ldots,t_N)$, i.e.,  $\P$ $\circ$ $(X_{t_1},\ldots,X_{t_N})^{-1}$ $=$  $\mu$, and $H(.|.)$ is the relative entropy between two probability measures defined by 
\begin{align}
\Hc(\P|\W) & = \; \left\{ 
\begin{array}{cc}
\int  \ln \frac{\d\P}{\d\W} \d\P, &\;  \mbox{ if }  \P  \ll  \W \\
\infty,  & \; \mbox{ otherwise }. 
\end{array}
\right.
\end{align}
Denoting by $\E_\P$ and $\E_\W$ the expectation under $\P$ and $\W$, we see that $\Hc(\P|\W)$ $=$  $\E_\P[\ln \frac{\d\P}{\d\W}]$ $=$ $\E_\W[ \frac{\d\P}{\d\W}\ln \frac{\d\P}{\d\W}]$ when $\P$ $\ll$ $\W$.  
Compared to the classical  Schr\"odinger bridge (SB) (see \cite{Leo14}, and the application to generative modeling in \cite{wangetal21}), which looks for a probability measure $\P$ that interpolates between an initial probability measure  and a  target probability  distribution at terminal time $T$,  here, we take into account via the constraint in $\Pc_\Tc^\mu$ the temporal dependence of the process observed at sequential times $t_1<\ldots<t_N$, and look for an entropic interpolation of the time series distribution. 
We call \eqref{SBP} the  Schr\"odinger bridge for  time series (SBTS) problem. 
 
\vspace{1mm}

Let us now  formulate  (SBTS) as a stochastic control problem following the well-known connection established for classical (SB) in \cite{daipra}.  Given  $\P$ $\in$ $\Pc(\Omega)$ with finite relative entropy $H(\P|\W)$ $<$ $\infty$, it is known by Girsanov's theorem that one can associate to $\P$ an $\F$-adapted $\R^d$-valued process 
$\alpha$ $=$ $(\alpha_t)$ with finite energy: $\E_{\P}[\int_0^T |\alpha_t|^2 \d t]$ $<$ $\infty$ such that 
\begin{align} \label{Palpha} 
 \ln \frac{\d\P}{\d\W}  &= \; \int_0^T \alpha_t. \d X_t - \frac{1}{2} \int_0^T |\alpha_t|^2 \d t, 
\end{align}
and  $X_t - \int_0^t \alpha_s \d s$, $0\leq t\leq T$, is a Brownian motion under $\P$. We then have 
\begin{align}
\Hc(\P|\W) & = \;  \E_{\P}\Big[  \frac{1}{2}\int_0^T   |\alpha_t|^2 \d t \Big]. 
\end{align}
Therefore, (SBTS) is reformulated equivalently in the language of  stochastic control as: 
\begin{align} \label{SBTScon}
\begin{cases}
\mbox{ Minimize over } \alpha \in \Ac,  \quad J(\alpha) \; = \;  \E_{\P} \Big[  \frac{1}{2}\int_0^T   |\alpha_t|^2 \d t \Big] \\
\mbox{ subject to } \;  \d X_t \; = \; \alpha_t \d t + \d W_t, \; X_0 = 0, \quad (X_{t_1},\ldots,X_{t_N}) \stackrel{\P}{\sim} \mu, 
\end{cases}
\end{align}
where $W$ is a Brownian motion under $\P$,  $\Ac$ is the set of $\R^d$-valued $\F$-adapted processes s.t. $\E_\P[\int_0^T |\alpha_t|^2 \d t]$ $<$ $\infty$, and $(X_{t_1},\ldots,X_{t_N}) \stackrel{\P}{\sim} \mu$ is the usual notation for $\P$ $\circ$ $(X_{t_1},\ldots,X_{t_N})^{-1}$ $=$  
$\mu$. In the sequel, when there is no ambiguity, we omit the reference  on $\P$ in  $\E$ $=$ $\E_{\P}$ and $\sim$ $=$ $\stackrel{\P}{\sim}$. 
We denote by $V_{_{SBTS}}$ the infimum of this stochastic control problem under joint distribution constraint: 
\begin{align}
V_{_{SBTS}} &:= \; \inf_{\alpha\in\Ac_\Tc^\mu} J(\alpha),
\end{align}
where $\Ac_\Tc^\mu$ is the set of controls $\alpha$ in $\Ac$ satisfying $(X_{t_1},\ldots,X_{t_N})$ $\sim$ $\mu$ with $X_t$ $=$ $\int_0^t \alpha_s \d s + W_t$. 
Our goal is to prove the existence of an optimal control $\alpha^*$ that can be explicitly derived, and then used to generate samples of the time series distribution $\mu$ via the probability measure $\P^*$ on $\Omega$, i.e., the simulation of the optimal diffusion process $X$ controlled by the drift  $\alpha^*$.

 \section{Solution to Schr\"odinger bridge for  time series} \label{sec:SB}

Similarly as for the classical Schr\"odinger bridge  problem, we assume that the target distribution $\mu$ admits a density with respect to the Lebesgue measure on $(\R^d)^N$, and by misuse of notation, we denote by $\mu(x_1,\ldots,x_N)$ this density function.  
Denote by $\mu_\Tc^W$ the distribution of the Brownian motion on $\Tc$, i.e. of $(W_{t_1},\ldots,W_{t_N})$, which admits a density given by (by abuse of language, we use the same notation for the measure and its density) 
\begin{align} \label{muW} 
\mu_\Tc^W(x_1,\ldots,x_N) &=  \prod_{i=0}^{N-1} \frac{1}{\sqrt{2\pi(t_{i+1}-t_i)}} \exp\Big( - \frac{|x_{i+1}-x_i|^2}{2(t_{i+1}-t_i)} \Big), 
\end{align}
for $(x_1,\ldots,x_N) \in  (\R^d)^N$ 
(with the convention that $t_0$ $=$ $0$, $x_0$ $=$ $0$). The measure $\mu$ is absolutely continuous with respect to $\mu_\Tc^W$, and we shall assume that its relative entropy is finite, i.e., 
\begin{align} \label{finitemu} 
\Hc(\mu|\mu_\Tc^W) \; = \; \int \ln \frac{\mu}{\mu_\Tc^W} \d\mu & < \; \infty. 
\end{align}

\vspace{1mm}

The solution to the (SBTS) problem is provided in the following theorem.

\begin{thm} \label{thmSB} 
The diffusion process $X_t$ $=$ $\int_0^t \alpha_s^* \d s + W_t$, $0\leq t\leq T$, with $\alpha^*$ defined as 
\begin{align}
\alpha_t^* &= \; \mra^*(t,X_t; (X_{t_i})_{t_i\leq t}), \quad 0 \leq t <T, 
\end{align}
with $\mra^*(t,x;\bolx_i)$, for $t$ $\in$ $[t_i,t_{i+1})$, $\bolx_i$ $=$ $(x_1,\ldots,x_i)$ $\in$ $(\R^d)^i$, $x$ $\in$ $\R^d$,  given by 
\begin{align} \label{driftopt} 
\mra^*(t,x;\bolx_i) &=  \nabla_x \ln \E_{\W}\Big[ 
\frac{\mu}{\mu_\Tc^W} (X_{t_1},\ldots,X_{t_N})
\big| \boX_{t_i} = \bolx_i, X_t = x \Big], 
\end{align}
where we set $\boX_{t_i}$ $=$ $(X_{t_1},\ldots,X_{t_i})$,  
induces a probability measure $\P^*$ $=$ $\frac{\mu}{\mu_\Tc^W}(X_{t_1},\ldots,X_{t_N})\W$, which solves the Schr\"odinger bridge time series problem. 
Moreover, we have 
\begin{align}
V_{_{SBTS}} & = \; \Hc(\P^*|\W) \; = \; \Hc(\mu|\mu_\Tc^W). 
\end{align}
\end{thm}
\begin{proof} First, observe that  $\E_\W[\frac{\mu}{\mu_\Tc^W}(X_{t_1},\ldots,X_{t_N})]$ $=$ $1$, and thus one can define a probability measure $\P^*$ $\ll$ $\W$ with density process
\begin{align}
Z_t &= \;  \E_{\W} \Big[ \frac{\d\P^*}{\d\W} \big| \Fc_t \Big] \; = \;  \E_{\W} \Big[ \frac{\mu}{\mu_\Tc^W}(X_{t_1},\ldots,X_{t_N}) \big| \Fc_t \Big], \quad 0 \leq t \leq T. 
\end{align}
Notice from the Markov and Gaussian properties of the Brownian motion  that  for $t$ $\in$ $[t_i,t_{i+1})$, $i$ $=$ $0,\ldots,N-1$, we have $Z_t$ $=$ $h_i(t,X_t;\boX_{t_i})$, where for a path $\bolx_i$ $=$ $(x_1,\ldots,x_i)$ $\in$ $(\R^d)^i$, 
$h_i(;\bolx_i)$ is defined on $[t_i,t_{i+1})\times\R^d$ by   
\begin{align}
h_i(t,x;\bolx_i) &= \; \E_{Y \sim \Nc(0,I_d)} \Big[ \frac{\mu}{\mu_\Tc^W}(\bolx_i,x+\sqrt{t_{i+1}-t} Y,\ldots,x+\sqrt{t_{N}-t} Y) \Big]
\end{align}
for $t$ $\in$ $[t_i,t_{i+1})$, $x$ $\in$ $\R^d$, and $\E_{Y \sim \Nc(0,I_d)}[.]$ is the expectation when $Y$ is distributed according to the Gaussian law $\Nc(0,I_d)$.  Moreover, by the law of conditional expectations, we have 
\begin{align}
h_i(t,x;\bolx_i) &= \; \E_{\W} \Big[ h_{i+1}(t_{i+1},X_{t_{i+1}};\bolx_i,X_{t_{i+1}}) \big| X_t = x \Big], 
\end{align}
with the convention that $h_N(t_N,x;\bolx_{N-1},x)$ $=$ $\frac{\mu}{\mu_\Tc^W}(x_1,\ldots,x_{N-1},x)$.  Therefore, for $i$ $=$ $0,\ldots,N-1$, and $\bolx$ $\in$ $(\R^d)^i$, the function $(t,x)$ $\mapsto$ $h_i(t,x;\bolx_i)$ is a strictly positive $C^{1,2}([t_i,t_{i+1})\times\R^d)$ $\cap$  
$C^{0}([t_i,t_{i+1}]\times\R^d)$ classical solution to the heat equation
\begin{align}
\Dt{h_i(.;\bolx_i)}  + \frac{1}{2} \Delta_x h_i(.;\bolx_i) &= \; 0, \quad \mbox{ on } [t_i,t_{i+1})\times\R^d, 
\end{align}
with the terminal condition: $h_i(t_{i+1},x;\bolx_i)$ $=$ $h_{i+1}(t_{i+1},x;\bolx_i,x)$ (here $\Delta_x$ is the Laplacian operator).  By applying  It\^o's formula to the martingale density process $Z$ of $\P^*$ 
under the Wiener measure $\W$, we derive 
\begin{align}
\d Z_t &= \; \nabla_x h_i(t,X_t;\boX_{t_i}) \d X_t, \\
&= \; Z_t \nabla_x \ln h_i(t,X_t;\boX_{t_i}) \d X_t,  \quad  t_i \leq t < t_{i+1}, 
\end{align}
for  $i$ $=$ $0,\ldots,N-1$. Thus, by defining the process $\alpha^*$ by $\alpha_t^*$ $=$ $\nabla_x \ln h_i(t,X_t;\boX_{t_i})$, for  $t$ $\in$ $[t_i,t_{i+1})$, $i$ $=$ $0,\ldots,N-1$, we have 
\begin{align}
\frac{\d\P^*}{\d\W} &= \; \exp\Big( \int_0^T \alpha_t^* \d X_t - \frac{1}{2} \int_0^T |\alpha_t^*|^2 \d t \Big), 
\end{align}
and by Girsanov's theorem, $X_t - \int_0^t \alpha_s^* \d s$ is a Brownian motion under $\P^*$.  On the other hand, by definition of $\P^*$, and Bayes formula,  we have for any bounded measurable function $\varphi$ on $(\R^d)^N$: 
\begin{align}
\E_{\P^*}\big[ \varphi(X_{t_1},\ldots,X_{t_N}) \big] &= \; \E_{\W} \Big[  \frac{\mu}{\mu_\Tc^W}(X_{t_1},\ldots,X_{t_N}) \varphi(X_{t_1},\ldots,X_{t_N}) \Big] \\
& = \; \int  \frac{\mu}{\mu_\Tc^W}(x_1,\ldots,x_{N}) \varphi(x_{1},\ldots,X_{N}) \mu_\Tc^W(x_1,\ldots,x_N) \d x_1\ldots\d x_N \\
& = \; \int  \varphi(x_{1},\ldots,X_{N}) \mu(x_1,\ldots,x_N) \d x_1\ldots\d x_N, 
\end{align}
which shows that $(X_{t_1},\ldots,X_{t_N}) \stackrel{\P^*}{\sim} \mu$. Moreover, by noting that 
\begin{align}
J_{}(\alpha^*) \; = \; \E_{\P^*} \Big[\int_0^T \frac{1}{2} |\alpha_t^*|^2 \d t \Big] \; = \; \E_{\P^*}\Big[ \ln \frac{\d\P^*}{\d\W} \Big] \; = \; \Hc(\P^*|\W) \; = \; \Hc(\mu|\mu_\Tc^W) < \infty,
\end{align}
where we used in the last inequality the fact that $(X_{t_1},\ldots,X_{t_N}) \stackrel{\P^*}{\sim} \mu$, this shows  in particular  that $\alpha^*$ $\in$ $\Ac_\Tc^\mu$.  
 
\vspace{1mm}

It remains to show that for any $\alpha$ $\in$ $\Ac_\Tc^\mu$ associated to a probability measure $\P$ $\ll$ $\W$ with density given by \eqref{Palpha}, i.e. $J_{}(\alpha)$ $=$ $\Hc(\P|\W)$, we have 
\begin{align} \label{inegalpha}
J(\alpha) 
& \geq \; \Hc(\mu|\mu_\Tc^W). 
\end{align}
For this, we write from Bayes formula and since  $W_t$ $=$ $X_t - \int_0^t \alpha_s \d s$ is a Brownian motion under $\P$ by Girsanov theorem:
\begin{align}
1 &= \; \E_{\W} \Big[  \frac{\mu}{\mu_\Tc^W}(X_{t_1},\ldots,X_{t_N})  \Big] \\
&= \; \E_{\P} \Big[  \exp\Big( \ln \frac{\mu}{\mu_\Tc^W}(X_{t_1},\ldots,X_{t_N})  - \int_0^T \alpha_t \d W_t - \frac{1}{2} \int_0^T |\alpha_t|^2 \d t \Big) \Big] \\
& \geq \; \exp \Big( \E_\P\Big[  \ln \frac{\mu}{\mu_\Tc^W}(X_{t_1},\ldots,X_{t_N})  - \int_0^T \alpha_t \d W_t - \frac{1}{2} \int_0^T |\alpha_t|^2 \d t \Big] \Big) \\
& = \; \exp\Big( \Hc(\mu|\mu_\Tc^W) - J(\alpha) \Big),
\end{align} 
where we use Jensen's inequality, and the fact that $(X_{t_1},\ldots,X_{t_N}) \stackrel{\P}{\sim} \mu$ in the last equality. This proves the required inequality \eqref{inegalpha}, and ends the proof.  
\end{proof}

\vspace{2mm}

\begin{rem}
The optimal drift of the Schr\"odinger bridge time series diffusion is in general path-dependent: it depends at given time $t$ not only on its current state $X_t$, but also on the past values $\boX_{\eta(t)}$ $=$ 
$(X_{t_1},\ldots,X_{\eta(t)})$, where $\eta(t)$ $=$ $\max\{ t_i: t_i \leq t\}$, and we have:
\begin{align}
\d X_t &= \; \mra^*(t,X_t;\boX_{\eta(t)}) \d t + \d W_t, \;\;\;  0 \leq t \leq T,  \;\; X_0 = 0.     
\end{align}
Moreover, the proof of the above theorem shows that this drift function is explicitly  given by 
\begin{align} \label{driftexpress1} 
 \mra^*(t,x;\bolx_i) &= \; \frac{\nabla_x h_i(t,x;\bolx_i)}{h_i(t,x;\bolx_i)}, \quad t \in [t_i,t_{i+1}), 
 \bolx_i \in (\R^d)^i, x \in \R^d, 
\end{align}
for $i$ $=$ $0,\ldots,N-1$, where 
\begin{align}
h_i(t,x;\bolx_i) &= \; \E_{Y \sim \Nc(0,I_d)} \Big[ \rho(\bolx_i,x+\sqrt{t_{i+1}-t} Y,\ldots,x+\sqrt{t_{N}-t} Y) \Big],    
\end{align}
with $\rho$ $:=$ $\frac{\mu}{\mu_\Tc^W}$ the density ratio. 
\end{rem}

\vspace{1mm}

The following result states an alternate representation of the drift function that will be useful in the next section for estimation. 

\begin{prop} \label{propdrift}
For $i$ $=$ $0,\ldots,N-1$, $t$ $\in$ $[t_i,t_{i+1})$, 
$\bolx_i$ $=$ $(x_1,\ldots,x_i)$ $\in$ $(\R^d)^i$, $x$ $\in$ $\R^d$, we have
\begin{align} \label{driftexpress2} 
    \mathrm{a}^*(t,x;\bolx_i) &= \;  \frac{1}{t_{i+1}-t} \frac{\mathbb{E}_{\mu}\left[ (X_{t_{i+1}}-x) F_i(t,x_i,x,X_{t_{i+1}}) \big| \boX_{t_i} = \bolx_i \right]}{\mathbb{E}_{\mu}\left[F_i(t,x_i,x,X_{t_{i+1}}) \big| \boX_{t_i} = \bolx_i \right]},
\end{align}
where 
\begin{align}
F_i(t,x_i,x,x_{i+1}) &=\; \exp\left(-\frac{|x_{i+1}-x|^2}{2(t_{i+1}-t)} + \frac{|x_{i+1}-x_{i}|^2}{2(t_{i+1}-t_i)}\right),
\end{align}
and $\E_\mu[\cdot]$ denotes the expectation under $\mu$. 
\end{prop}
 \begin{proof}
Fix $i$ $\in$ $\llbracket 0,N-1\rrbracket$, and $t$ $\in$ $[t_i,t_{i+1})$. From the expression of $\mu_\Tc^W$ in \eqref{muW}, we have
\begin{align} \label{aStarNewExpression}
&\mathbb{E}_{\mathbb{W}}\left[ \frac{\mu}{\mu_{\mathcal{T}}^W}(X_{t_1},\cdots,X_{t_N}) \Big\vert (X_{t_1},\cdots,X_{t_i})=(x_1,\cdots,x_i), \ X_t=x\right]\\
&=C \int_{} \frac{\mu}{\mu_{\mathcal{T}}^W}(x_1,\cdots,x_N) \exp\left(-\frac{|x_{i+1}-x|^2}{2(t_{i+1}-t)}\right) \prod_{j=i+1}^{N-1} \exp\left(-\frac{|x_{j+1}-x_j|^2}{2\Delta t_{j}}\right) \d x_{i+1}\cdots \d x_N\\
&=C\int_{} F_i(t,x_i,x,x_{i+1}) \frac{\mu(x_1,\ldots,x_N)}{\mu_i(x_1,\ldots,x_i)} \d x_{i+1}\cdots \d x_N 
\; = \; C \E_\mu\Big[ F_i(t,x_i,x,X_{t_{i+1}}) \big| \boX_{t_i} = \bolx_i  \Big],
\end{align}
where $C$ is a constant varying from line to line and depending only on $t$ and $\bolx_i$ $=$ $(x_1,\cdots,x_i)$, but not on $x$, 
and $\mu_i$ is the density of $(X_{t_1},\cdots,X_{t_i})$ under $\mu$, \textit{i.e.},
\begin{align}
    \mu_{i}(x_1,\cdots,x_i) = \int_{} \mu(x_1,\cdots,x_{N}) \d x_{i+1}\cdots \d x_N.
\end{align}
By plugging the new expression \eqref{aStarNewExpression} into $\mathrm{a}^*$ and differentiating with respect to $x$, we then get
\begin{align}
    \mathrm{a}^*(t,x;\bolx_i) = \frac{1}{t_{i+1}-t} \frac{\mathbb{E}_{\mu}\left[ (X_{t_{i+1}}-x) F_i(t,x_i,x,X_{t_{i+1}}) \big| \boX_{t_i} = \bolx_i \right]}{\mathbb{E}_{\mu}\left[F_i(t,x_i,x,X_{t_{i+1}}) \big| \boX_{t_i} = \bolx_i \right]}.
\end{align}
\end{proof}

\begin{rem}
 In the case where $\mu$ is the distribution arising from a Markov chain, i.e., in the form 
 $\mu(\d x_1,\ldots,x_N)$ $=$ $\prod_{i=1}^{N-1} \nu_i(x_i,\d x_{i+1})$, for some transition kernels $\nu_i$ on $\R^d$, then 
 the conditional expectations in \eqref{driftexpress2} will depend on the past values $\boX_{t_i}$ $=$ $(X_{t_1},\ldots,X_{t_i})$ 
 only via the last value $X_{t_i}$. 
\end{rem}

\section{Generative learning}

From Theorem \ref{thmSB}, we can run an Euler scheme for simulating the Schr\"odinger bridge diffusion, and then samples of the target distribution $\mu$. For that purpose, we need an accurate estimation of the drift terms, i.e., of the functions $\mra_i^*$, for $i$ $=$ $0,\ldots,N-1$. We propose several estimation methods. In the sequel, for a probability measure $\nu$ on $(\R^d)^N$, we denote by $\E_\nu[\cdot]$ the expectation under the distribution $\nu$. 

\subsection{Drift estimation}

\paragraph{Estimation of the density ratio.} This method follows the idea in \cite{wangetal21}. Denote by $\rho$ $=$ $\frac{\mu}{\mu_\Tc^W}$ the density ratio, and observe that the log-density ratio $\ln\rho$ minimizes over functions 
$r$ on $(\R^d)^N$ the logistic regression function
\begin{align}
 L_{logistic}(r) &= \; \E_{\mu}\Big[ \ln \big(1 + \exp(-r(\boX) \big)\Big]  
 + \E_{\mu_\Tc^W}\Big[ \ln \big(1 + \exp(r(\boX) \big) \Big].  
\end{align}
Therefore, given data samples $\boX^{(m)}$ $=$ $(X_{t_1}^{(m)},\ldots,X_{t_N}^{(m})$, $m$ $=$ $1,\ldots,M$ from $\mu$, and using samples $\boY^{(m)}$ from $\mu_\Tc^W$, we estimate the density ratio $\rho$ by 
\begin{align}
 \hat\rho &= \;    \exp(r_{\hat\theta})
\end{align}
where $r_{\hat\theta}$ is the neural network that minimizes the empirical logistic loss function: 
\begin{align}
\theta \; \mapsto \; \frac{1}{M} \sum_{m=1}^M \ln \big(1 + \exp(-r(\boX^{(m)}) \big) + \ln \big(1 + \exp(r(\boY^{(m)}) \big).     
\end{align}
By writing from \eqref{driftexpress1} the Schr\"odinger drift $\mra^*$ as
\begin{align}
\mra^*(t,x;\bolx_i) 
& = \; \frac{\E_{Y \sim \Nc(0,I_d)} 
\Big[ \rho \nabla_x \rho (\bolx_i,x+\sqrt{t_{i+1}-t} Y,\ldots,x+\sqrt{t_{N}-t} Y) \Big] }
{ \E_{Y \sim \Nc(0,I_d)} \Big[ \rho(\bolx_i,x+\sqrt{t_{i+1}-t} Y,\ldots,x+\sqrt{t_{N}-t} Y) \Big]}, \label{aistar}
\end{align}
for $t$ $\in$ $[t_i,t_{i+1})$, we obtain an estimator of $\mra^*$ by plugging into \eqref{aistar} the estimate $\hat\rho$, and $r_{\hat\theta}$ of 
$\rho$, and  $\ln\rho$, and then computing the expectation with Monte-Carlo approximations from samples in $\Nc(0,I_d)$. 
Notice that this method is very costly as it requires in addition to the training of the neural networks for estimating the density ratio, another Monte-Carlo sampling for estimating finally the drift.

\paragraph{Kernel estimation of the drift.} In order to overcome the computational issue of the above estimation method, we propose an alternative approach that relies on the representation  of the drift term in Proposition \ref{propdrift}.  
Indeed, the key feature of the formula \eqref{driftexpress2} is that it involves (conditional) expectations under the target distribution $\mu$, which is amenable to direct estimation using data samples. For the approximation of the conditional expectation, we can then use the classical kernel methods.

From data samples $\boX^{(m)}$ $=$ $(X_{t_1}^{(m)},\ldots,X_{t_N}^{(m)})$, $m$ $=$ $1,\ldots,M$ from $\mu$, the Nadaraya-Watson estimator of the drift function is given by 
\begin{align} \label{estimdrift} 
\hat\mra(t,x;\bolx_i) & = \;  \frac{1}{t_{i+1}-t} 
\frac{\Sum_{m=1}^M (X_{t_{i+1}}^{(m)} - x) F_i(t,X_{t_i}^{(m)},x,X_{t_{i+1}}^{(m)}) \Prod_{j=1}^i K_h(x_j-X_{t_j}^{(m)})} 
{\Sum_{m=1}^M  F_i(t,X_{t_i}^{(m)},x,X_{t_{i+1}}^{(m)}) \Prod_{j=1}^i K_h(x_j-X_{t_j}^{(m)})}, 
\end{align}
for $t$ $\in$ $[t_i,t_{i+1})$, $\bolx_i$ $\in$ $(\R^d)^i$, $x$ $\in$ $\R^d$, $i$ $=$ $0,\ldots,N-1$, where $K_h$ is a kernel, i.e., a non-negative real-valued integrable and symmetric function on $\R^d$, with bandwith $h$ $>$ $0$. 
Common kernel function is the Gaussian density function, but we shall use here for lower time complexity reason, the quartic kernel $K_h(x)$ $=$ $\frac{1}{h}K(\frac{x}{h})$ with 
\begin{align}
 K(x) &=   (1 - |x|^2) 1_{|x|\leq 1}.  
\end{align}

\paragraph{LSTM  network approximation of the path-dependent drift.}
The conditional expectations in the numerator and denominator of the drift terms can alternately be approximated by neural networks. In order to achieve this, we need a neural network architecture that fits well with the path-dependency of the drift term, i.e., the 
{\it a priori} non Markov feature of the data time series distribution $\mu$. We shall then consider a combination of feed-forward and LSTM (Long Short Term Memory) neural network. 
\begin{figure}[H]
    \centering
    \includegraphics[width=14cm,height=3cm]{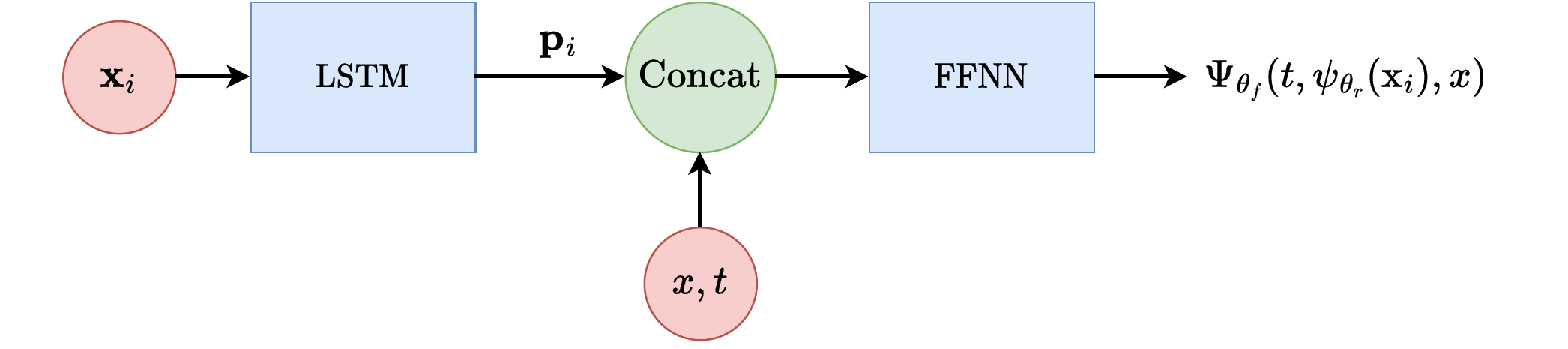}
    \caption{Architecture of the neural network}
    \label{fig:my_label}
\end{figure}

For $i$ $=$ $0,\ldots,N-1$, the conditional expectation function  in the numerator: 
\begin{align}
& \quad  (t,\bolx_i,x) \in [t_i,t_{i+1})\times(\R^d)^i\times\R^d \\
\longmapsto & \quad  \E_\mu \Big[ (X_{t_{i+1}}-X_t) F(t,X_{t_i},X_t,X_{t_{i+1}}) \big| (\boX_{t_i},X_t) = (\bolx_i,x) \Big],  
\end{align}
is approximated by  $\Psi_{\theta_f}(t,p_{t_{i}},x)$, where $\Psi_{\theta_f}$ $\in$ $NN_{d+k+1,d}$ is a feed-forward neural network with input dimension $1+k+d$, and output dimension $d$, and $p_{t_{i}}$ is an output vector of dimension $k$ from an LSTM network, i.e., 
$p_{t_i}$ $=$ $\psi_{\theta_r}(\bolx_i)$ with $\psi_{\theta_r}$ $\in$ ${\rm LSTM}_{i,d,k}$ at time $t_i$. This neural network is trained by minimizing over the parameter $\theta$ $=$ $(\theta_f,\theta_r)$ the quadratic loss function
\begin{align}
L(\theta) &= \; \sum_{i=0}^{N-1} \hat\E \Big| (X_{t_{i+1}}-X) F(\tau,X_{t_i},X,X_{t_{i+1}}) -   
\Psi_{\theta_f}(\tau,\psi_{\theta_r}(\boX_{t_{i}}),X) \Big|^2. 
\end{align}
 Here $\hat\E$ is the empirical loss expectation where $(X_{t_1},\ldots,X_{t_N})$ are sampled from the data distribution $\mu$, 
 $\tau$ is sampled according to an uniform law on $[t_i,t_{i+1})$, and $X$ is sampled e.g. from a Gaussian law with mean $X_{t_i}$ for $i$ $=$ $0,\ldots,N-1$.  The conditional expectation function in the denominator is similarly approximated.

The output of this neural network training yields an approximation:
\begin{align}
 t \in [0,T), \bolx \in (\R^d)^{\eta(t)}, x \in \R^d &  \longmapsto \; \hat\mra(t,x;\bolx),
\end{align}
of the drift term function, which is then used for generating samples $(X_{t_1},\ldots,X_{t_N})$ of $\mu$ from the simulation of the diffusion:
\begin{align} \label{SDESBTS} 
\d X_t &= \; 
\hat\mra(t,X_t;\boX_{\eta(t)}) \d t + \d W_t, \; X_0 = 0. 
\end{align}

\subsection{Schr\"odinger bridge time series algorithm}

From the estimator $\hat\mra$ of the drift path-dependent function, we can now simulate the SB SDE \eqref{SDESBTS} by an Euler scheme. Let $N^\pi$ be the number of uniform time steps between two consecutive observation dates $t_i$ and  $t_{i+1}$, for $i$ $=$ $0,\ldots,N-1$, and $t_{k,i}^\pi$ $=$ $t_i$ $+$ $\frac{k}{N^\pi}$, $k$ $=$ $0,\ldots,N^\pi -1$, the associated time grid. 
The pseudo-code  of the Schr\"odinger bridge time series (SBTS)  algorithm is  described in Algorithm \ref{AlgoSBTS}.

\vspace{3mm}

\begin{algorithm2e}[H] 
\DontPrintSemicolon 
\SetAlgoLined 
\vspace{1mm}
{\bf Input}: data samples of time series $(X_{t_1}^\mathrm{(m)},\cdots,X_{t_N}^\mathrm{(m)})$, $m=1,\ldots,M$, and $N^\pi$.  

{\bf Initialization}: initial state $x_0=0$; \\
\For{$i$ $=$ $0,\ldots,N-1$}
{Initialize state $y_0=x_i$; \\
\For{$k$ $=$ $0,\ldots,N^\pi-1$} 
 {Compute $\hat\mra(t^\pi_{k,i},y_k;\bolx_i)$ e.g. by kernel estimator \eqref{estimdrift};\\
Sample $\varepsilon_k \in \mathcal{N}(0,1)$ and compute 
\begin{align}
y_{k+1} &= \;  y_k + \frac{1}{N^\pi} \hat\mra(t^\pi_{k,i},y_k;\bolx_i)  + \frac{1}{\sqrt{N^\pi}} \varepsilon_k,
\end{align}
}
{
Set $x_{i+1}=y_{N^\pi}$. 
}
}
{\bf Return}: $x_1,\cdots,x_N$
\caption{SBTS  Simulation  \label{SBSimulation} } \label{AlgoSBTS}
\end{algorithm2e}

\section{Numerical experiments}

In this section, we demonstrate the effectiveness of our SBTS algorithm on several examples of time series models, as well as on real data sets for an application to deep hedging.   
The algorithms are performed on a computer with the following characteristics:
Intel(R) i7-7500U CPU @ 2.7GHz, 2 Core(s).

\subsection{Evaluation metrics}

In addition to visual plot of data vs generator samples path, we use some metrics to evaluate the accuracy of our generators:  
 
\begin{itemize}
\item {\it Marginal metrics} for quantifying how well are the marginal distributions from the generated samples compared to the data ones.  These include 
\begin{itemize}
\item Classical statistics like mean, $95\%$ and $5\%$ percentiles
\item Kolmogorov-Smirnov test: we compute the $p$-valued, and when $p$ $>$ $\alpha$ (usually $5\%$), we do not reject the null-hypothesis (generator came from data of reference distribution) 
\end{itemize}
\item {\it Temporal dynamics} metrics for quantifying the ability of the generator to capture the time structure of the time series data. We compute 
the empirical distribution of the quadratic variation: $\sum_i |X_{t_{i+1}}-X_{t_i}|^2$.
\item  {\it  Correlation structure} for evaluating the ability of the generator to capture the multi-dimensional structure of the time series. We shall compare the empirical covariance or correlation  matrix induced by the generator SBTS and the ones from the data samples.
\end{itemize}

\subsection{Toy autoregressive  model of time series}

We consider the following toy autoregressive (AR) model: 
\begin{equation}
\begin{cases}
    X_{t_1} & = \;  b +  \varepsilon_1,     \\
    X_{t_2} & =  \; \beta_1 X_{t_1} + \varepsilon_2,  \\
    X_{t_3} & = \;  \beta_2 X_{t_2} + \sqrt{|X_{t_1}|} + \varepsilon_3,  
    \end{cases}
\end{equation}
where the noises $\varepsilon_i$ $\sim$ $\mathcal{N}(0,\sigma_i^2)$, $i$ $=$ $1,\ldots,3$ are mutually independent. The model parameters are $b=0.7$, $\sigma_1=0.1$, $\sigma_2=\sigma_3=0.05$ and $\beta_1=\beta_2=-1$.

We use samples of size $M$ $=$ $1000$ for simulated data of the AR model. The drift of the SBTS diffusion is 
estimated with a kernel of bandwith $h$ $=$ $0.05$, and simulated from euler scheme with $N^\pi$ $=$ $100$. The runtime for generating $500$ paths of SBTS is $8$ seconds.


In Figure \ref{fig:plotAR}, we plot the empirical distribution of each pair $(X_{t_i},X_{t_j})$ from the AR model, and from the generated SBTS. We also show the marginal empirical distributions.  
Table \ref{table:pAR} presents the marginal metrics for the AR model and generator ($p$-value and percentiles at level $5\%$ and $95\%$). In Table \ref{table:corAR}, we give the difference between the empirical correlation from generated samples and AR data samples.

\begin{figure}[H]
    \centering
    \includegraphics[width=4.75cm,height=4.75cm]{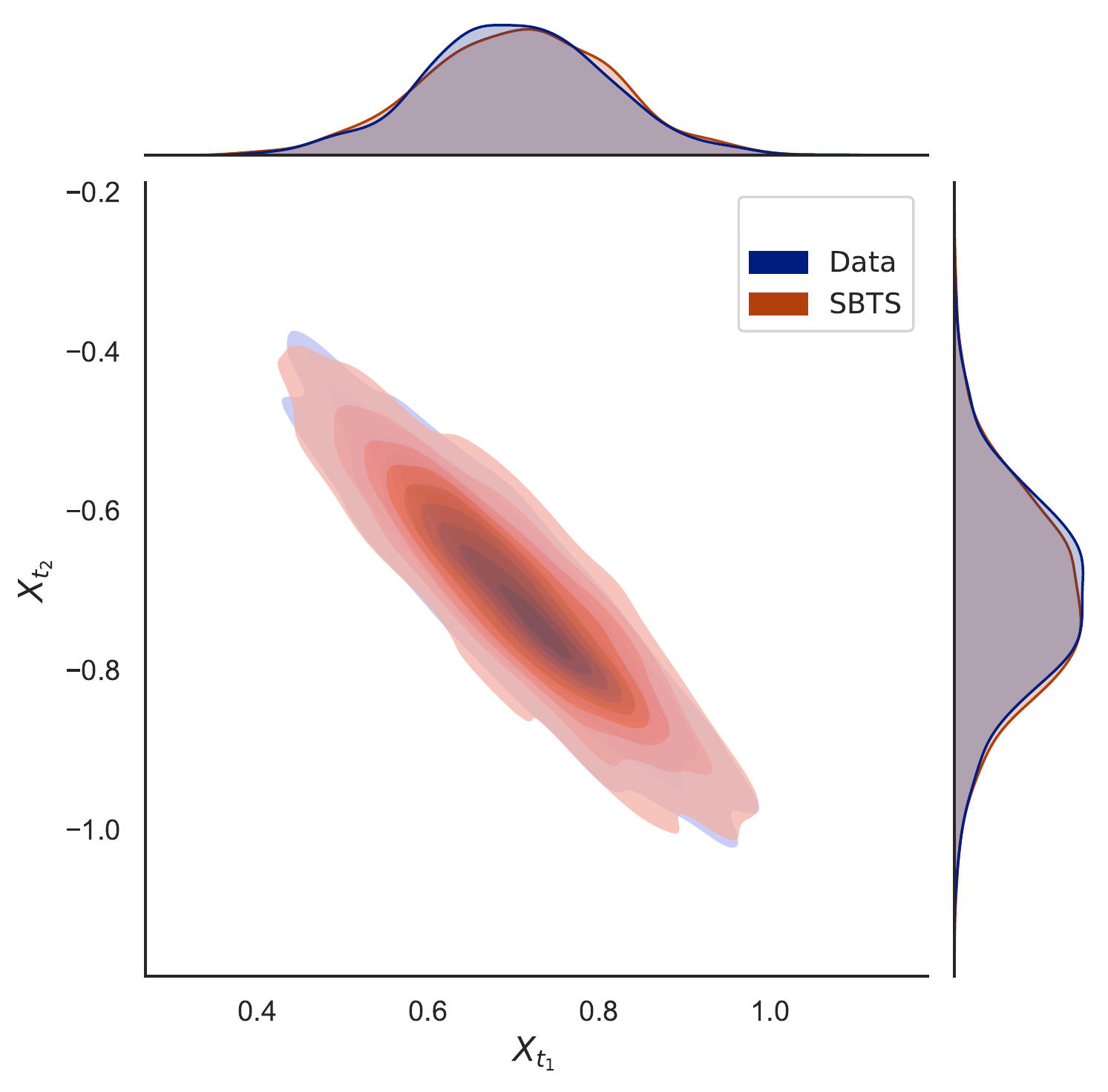}\includegraphics[width=4.75cm,height=4.75cm]{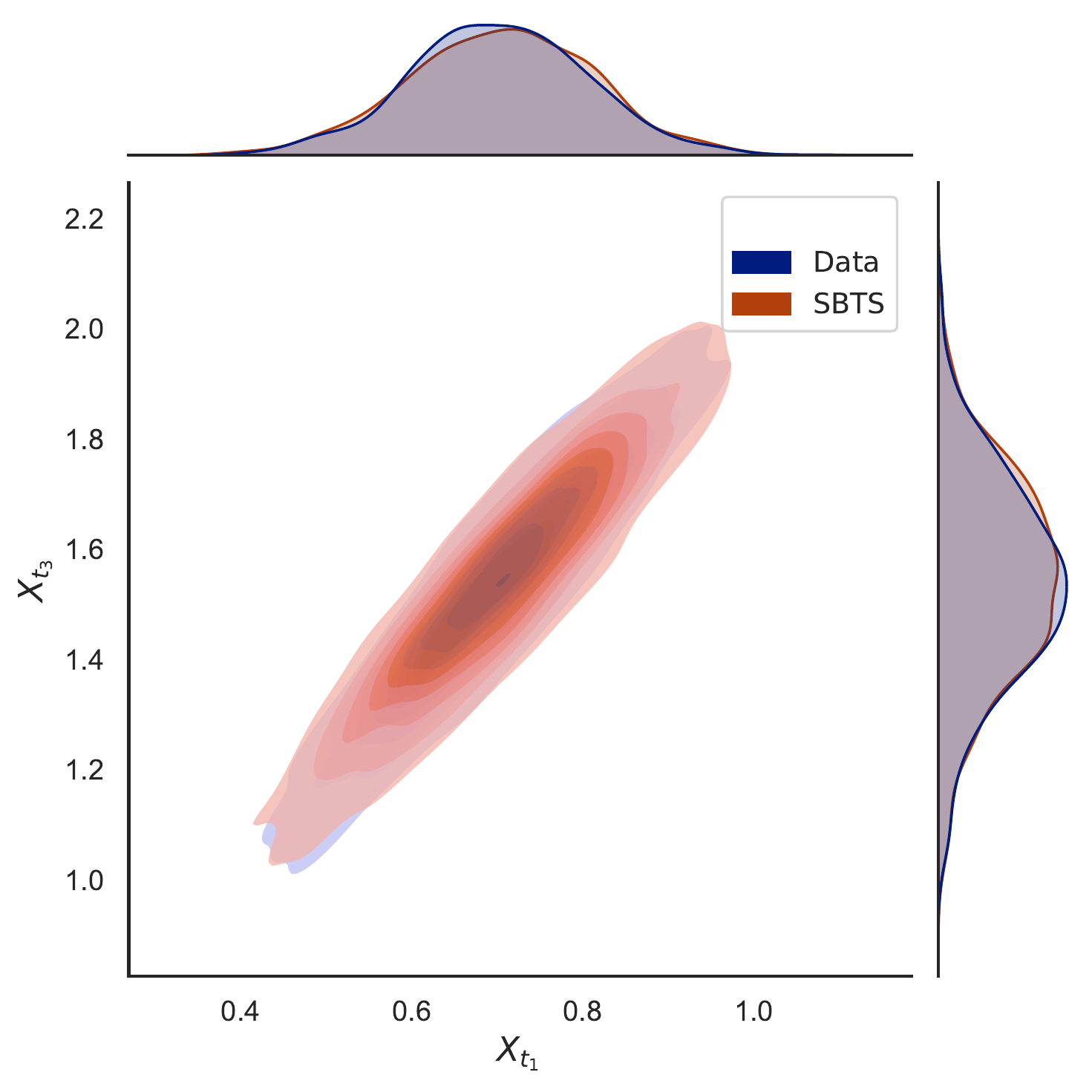} 
    \includegraphics[width=4.75cm,height=4.75cm]{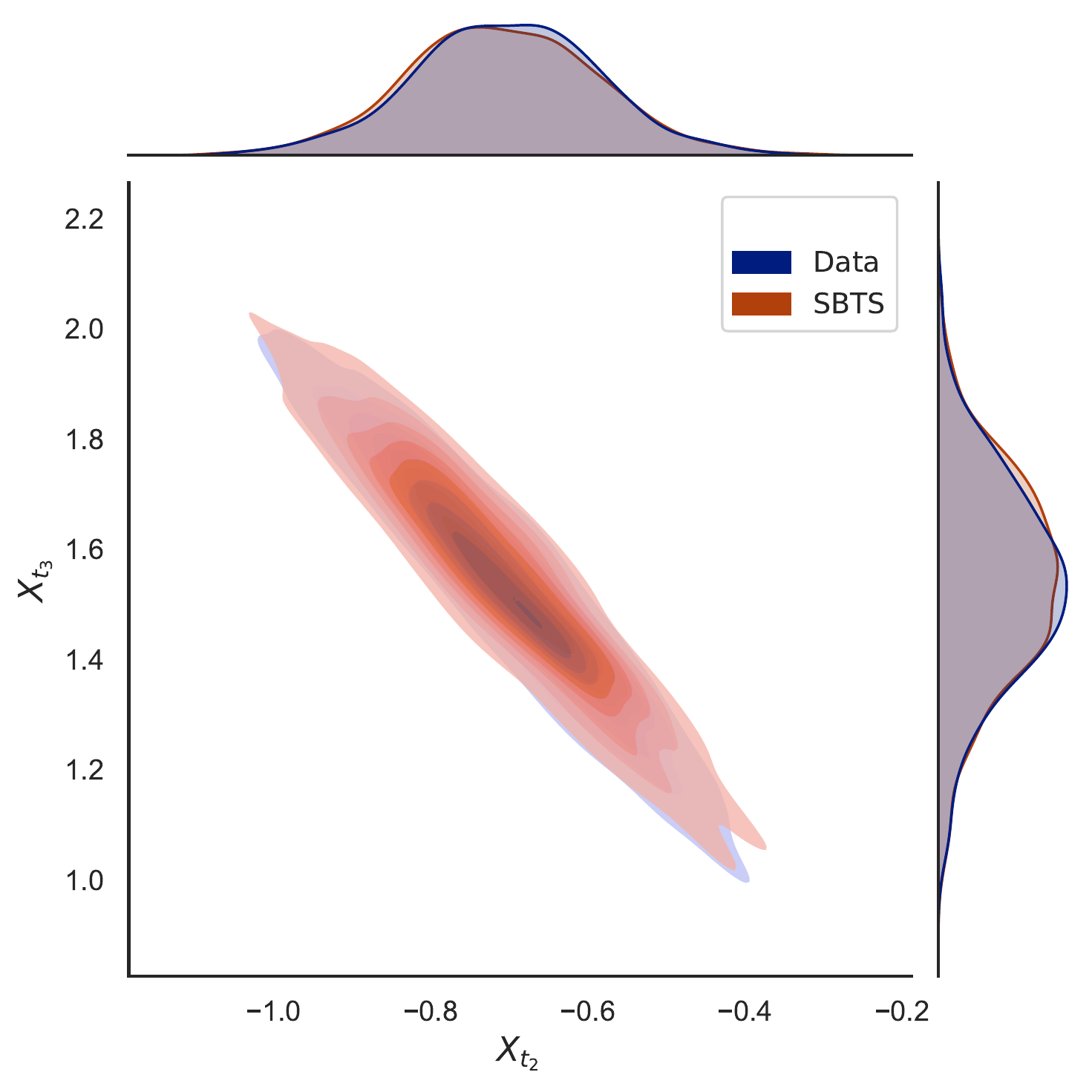}
    \caption{\footnotesize{Comparison between the true and generated distribution for each couple $(X_{t_i},X_{t_j})$ with $i,j\in\llbracket1,3 \rrbracket$ with $i\ne j$}}
    \label{fig:plotAR}
\end{figure}

\begin{table}[H]
\centering
\begin{tabular}{c|l|l|l|l|l|}
\cline{2-6}
\multicolumn{1}{l|}{}            & \multicolumn{1}{c|}{\textbf{p-value}} & \multicolumn{1}{c|}{\textbf{$q_{5}$}} & \multicolumn{1}{c|}{\textbf{$\tilde{q}_{5}$}} & \multicolumn{1}{c|}{\textbf{$q_{95}$}} & \multicolumn{1}{c|}{\textbf{$\tilde{q}_{95}$}} \\ \hline
\multicolumn{1}{|c|}{$X_{t_1}$} & 0.98    & 0.535  & 0.528  & 0.855  & 0.861 \\ 
\hline
\multicolumn{1}{|c|}{$X_{t_2}$} & 0.74    & -0.873 & -0.861 & -0.516 & -0.514 \\ 
\hline
\multicolumn{1}{|c|}{$X_{t_3}$} & 0.90    & 1.243  & 1.251  & 1.808  & 1.793 \\ 
\hline
\end{tabular}
\caption{\footnotesize{Marginal metrics for AR model and generator ($\tilde q$ for percentile)}}
\label{table:pAR}
\end{table}

\begin{table}[H]
\centering
\begin{tabular}{l|l|l|l|}
\cline{2-4}
                        & $X_{t_1}$ & $X_{t_2}$ & $X_{t_3}$     \\ \hline
\multicolumn{1}{|l|}{$X_{t_1}$} & 0         & 0.014 & -0.01 \\ \hline
\multicolumn{1}{|l|}{$X_{t_2}$} & 0.014     & 0     & 0.013 \\ \hline
\multicolumn{1}{|l|}{$X_{t_3}$} & -0.01 & 0.013 & 0     \\ \hline
\end{tabular}
\caption{\footnotesize{Difference between empirical correlation from generated samples and reference samples}}
\label{table:corAR}
\end{table}

\subsection{GARCH Model}

We consider a GARCH model:
\begin{equation}
\begin{cases}
    X_{t_{i+1}} &= \; \sigma_{t_{i+1}} \varepsilon_{t_{i+1}}\\
    \sigma^2_{t_{i+1}} &= \;  \alpha_0 + \alpha_1  X^2_{t_i} + \alpha_2 X^2_{t_{i-1}}, \quad i=1,\ldots,N, 
 \end{cases}   
\end{equation}
with $\alpha_0 = 5$, $\alpha_1=0.4$, $\alpha_2=0.1$, and the noises $\varepsilon_{t_i}$ $\sim\mathcal{N}(0,0.1)$, 
$i$ $=$ $1,\ldots,N$, are i.i.d. The size of the time series is $N$ $=$ $60$. 

The hyperparameters for the training and generation of SBTS are $M$ $=$ $1000$, $N^\pi$ $=$ $100$, and a bandwith for the kernel estimation $h$ $=$ $0.2$ (larger than for the AR model since by nature the GARCH process is more "volatile"). The runtime for generating $1000$ paths is $120$ seconds. 


In Figure \ref{fig:garch_paths}, we plot four sample paths of the GARCH process to be compared with four sample paths of the SBTS. Figure \ref{fig:garch1N} represents samples plot of the joint distribution between $(X_{t_1}$ and the terminal value $X_{t_N})$ of the time series. Figure \ref{fig:toy_model_2_pvalue_dist} provides some metrics to evaluate the performance of SBTS. On the left, we represent the $p$-value for each of the marginals of the generated SBTS. On the right, we compute for each marginal index $i$ $=$ $1,\ldots,N=60$, the difference between $\rho_i$ and 
$\hat\rho_i$ where $\rho_i$ (resp. $\hat\rho_i$)  is the sum over $j$ of the empirical correlation between  $X_{i}$ and $X_{t_j}$ from GARCH (resp. generated SBTS), and plot its mean and standard deviation.

\begin{figure}[H]
    \centering
    \includegraphics[width=12cm,height=4cm]{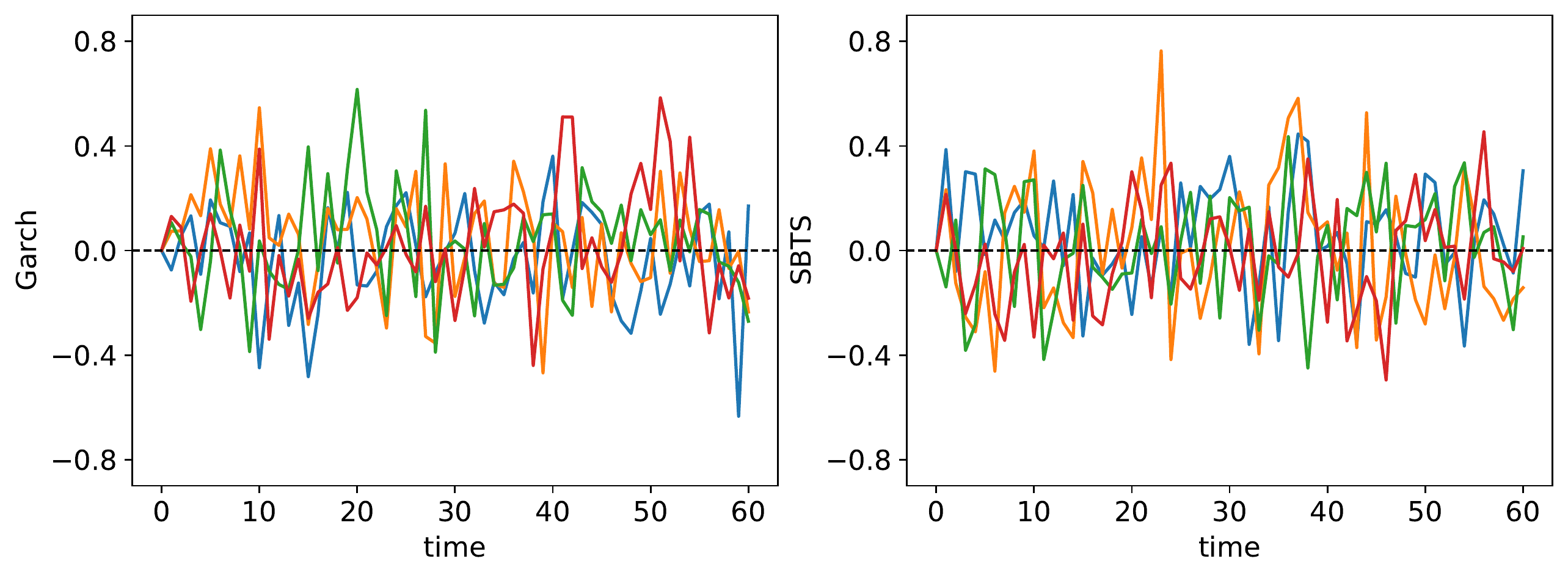}
    \caption{\footnotesize{Samples path of reference GARCH (left) and generator SBTS (right)}}
    \label{fig:garch_paths}
\end{figure}

\begin{figure}[H] 
    \centering
    \includegraphics[width=7cm,height=4cm]{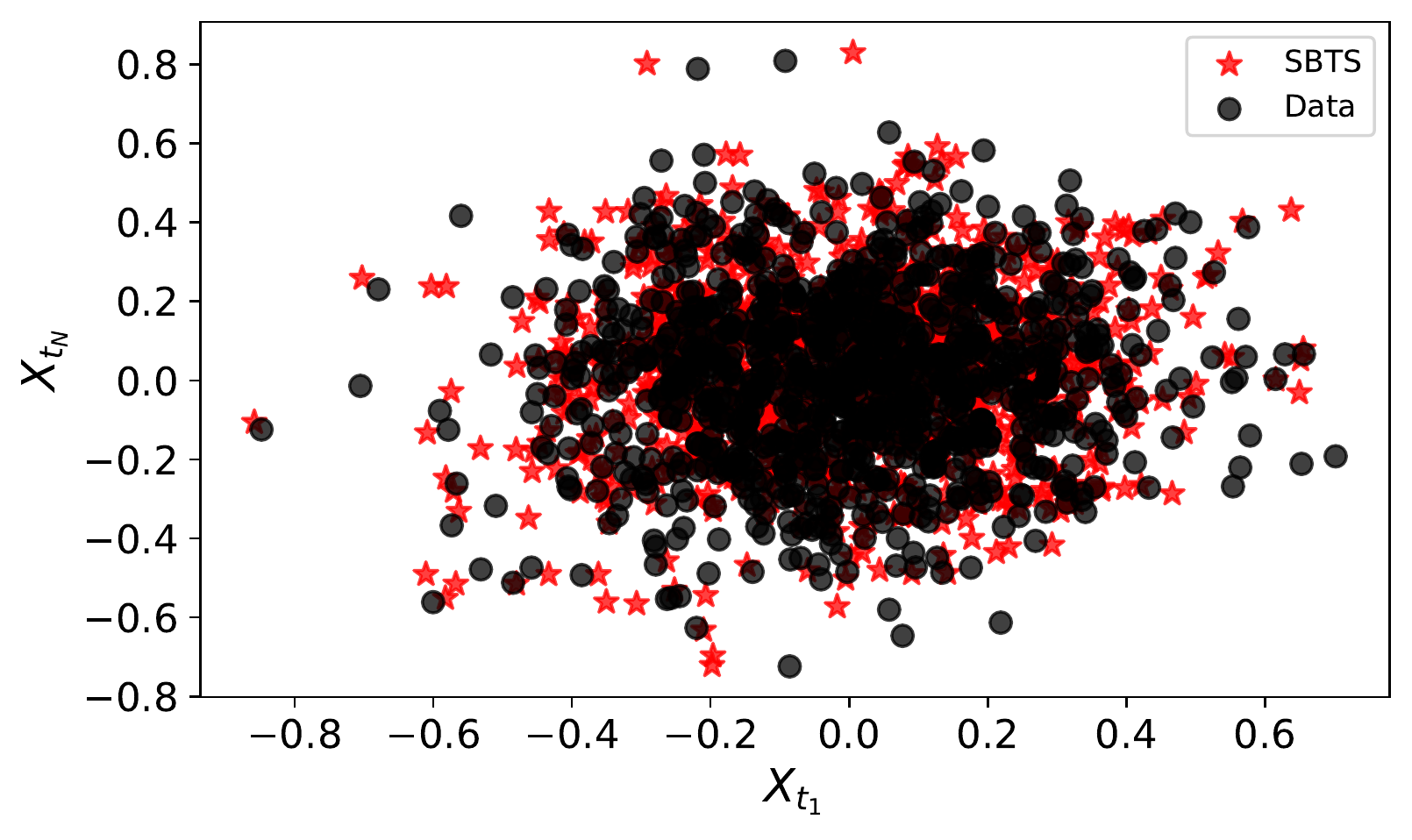}
    \caption{\footnotesize{Samples plot of the joint distribution $(X_{t_1},X_{t_N})$}} 
    \label{fig:garch1N}
\end{figure}

\begin{figure}[H]
    \centering
    \includegraphics[width=7cm,height=4cm]{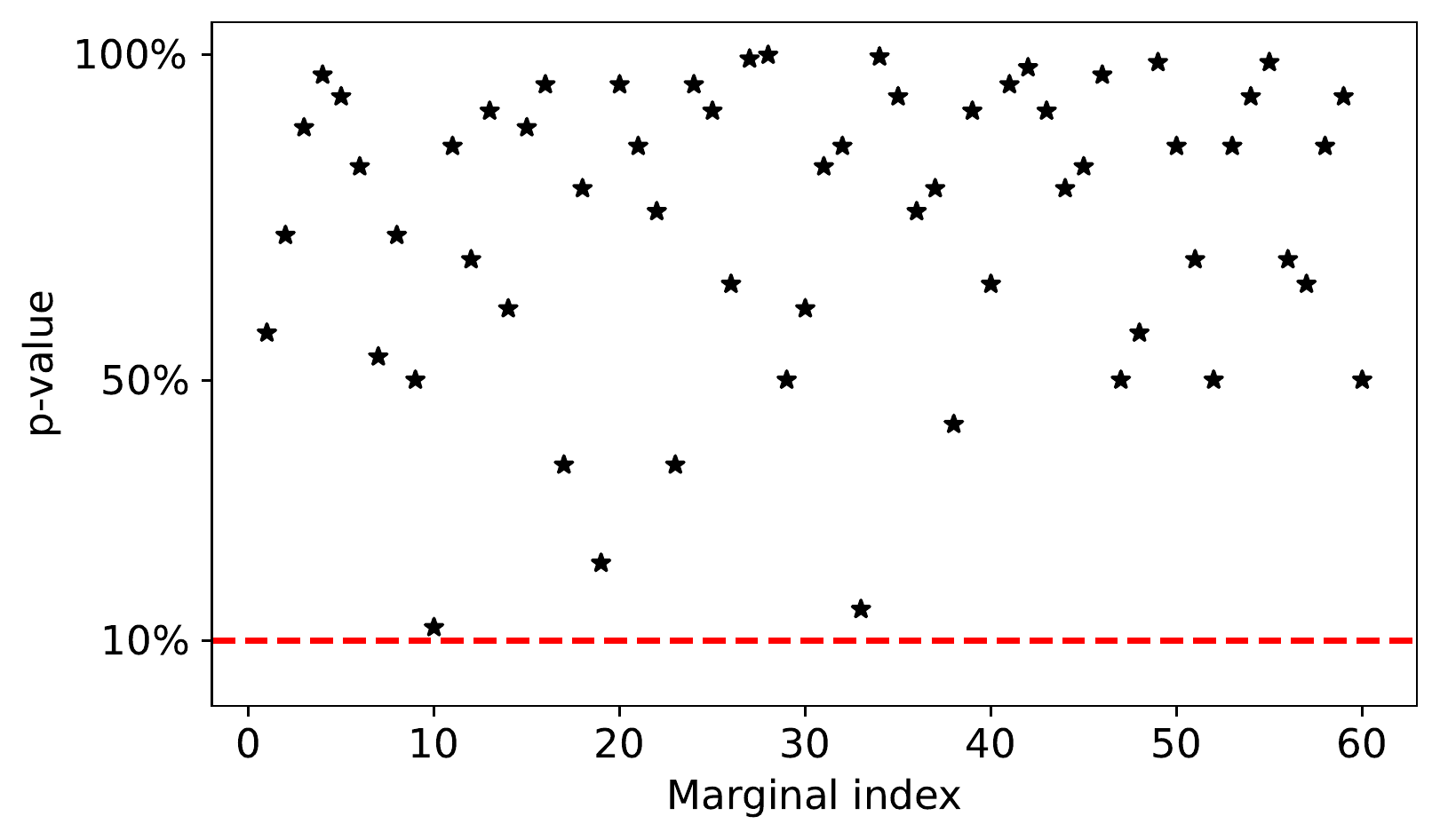}
    \includegraphics[width=7cm,height=4cm]{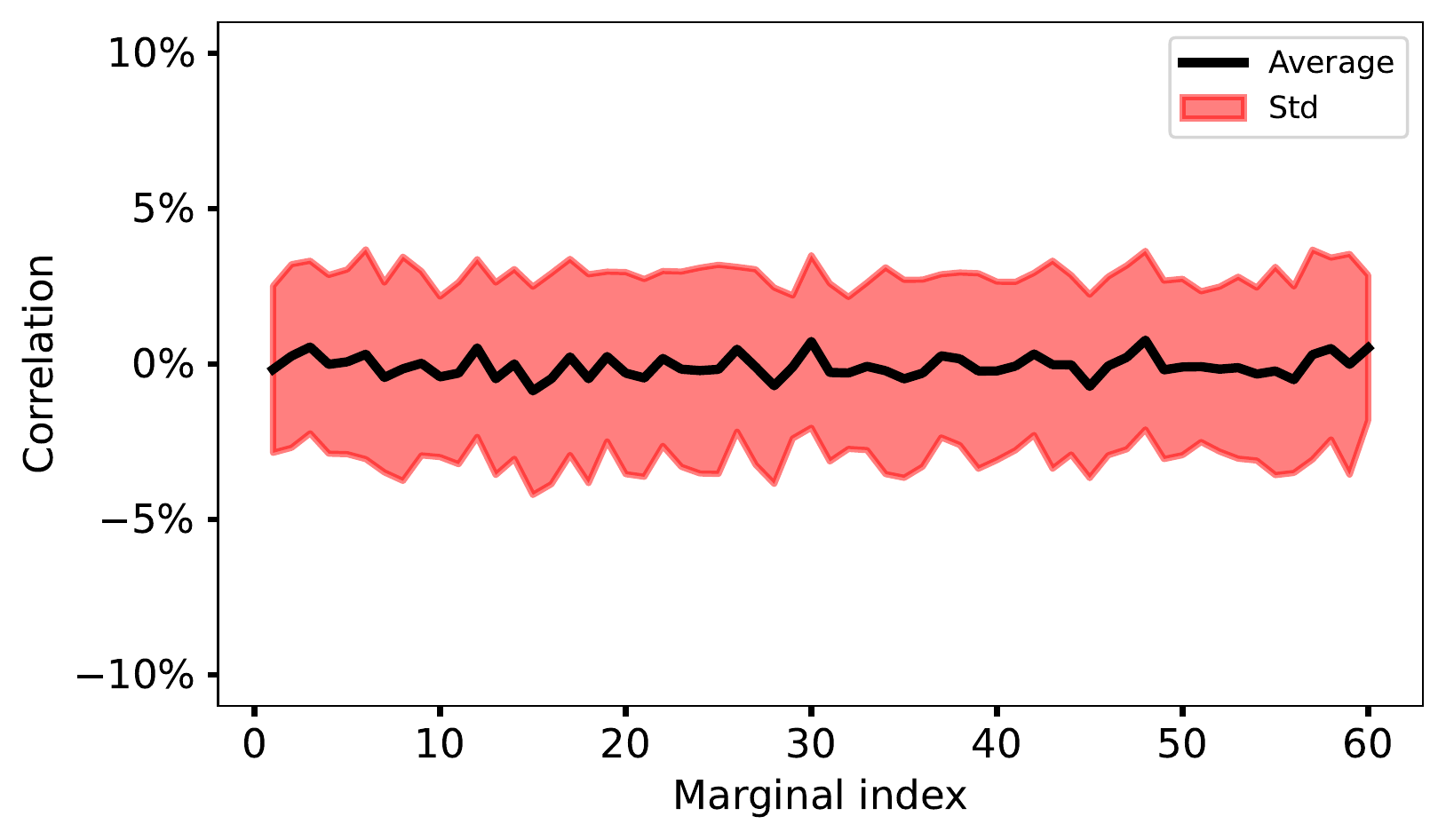}
     \caption{\footnotesize{{\it Left}: $p$-value for the marginals $X_{t_i}$. {\it Right}: Difference between the term-by-term empirical correlation from generated samples and reference samples. 
     }}
    \label{fig:toy_model_2_pvalue_dist}
\end{figure}

\subsection{Fractional Brownian Motion}

We consider a fractional Brownian motion (FBM) with Hurst index $H$ that measures the roughness of this Gaussian process. We plot in Figure  \ref{fig:FBM H=0.2} four samples path of FBM with $H$ $=$ $0.2$, and samples paths generated by SBTS. The generator is trained with $M$ $=$ $1000$ sample paths, and 
the hyperparameters used for the simulation are $N^\pi$ $=$ $100$, with bandwith $h$ $=$ $0.05$ for the kernel estimation of the Schr\"odinger drift. The runtime for $1000$ paths is $120$ seconds.



\begin{figure}[H]
    \centering
    \includegraphics[width=13cm,height=5cm]{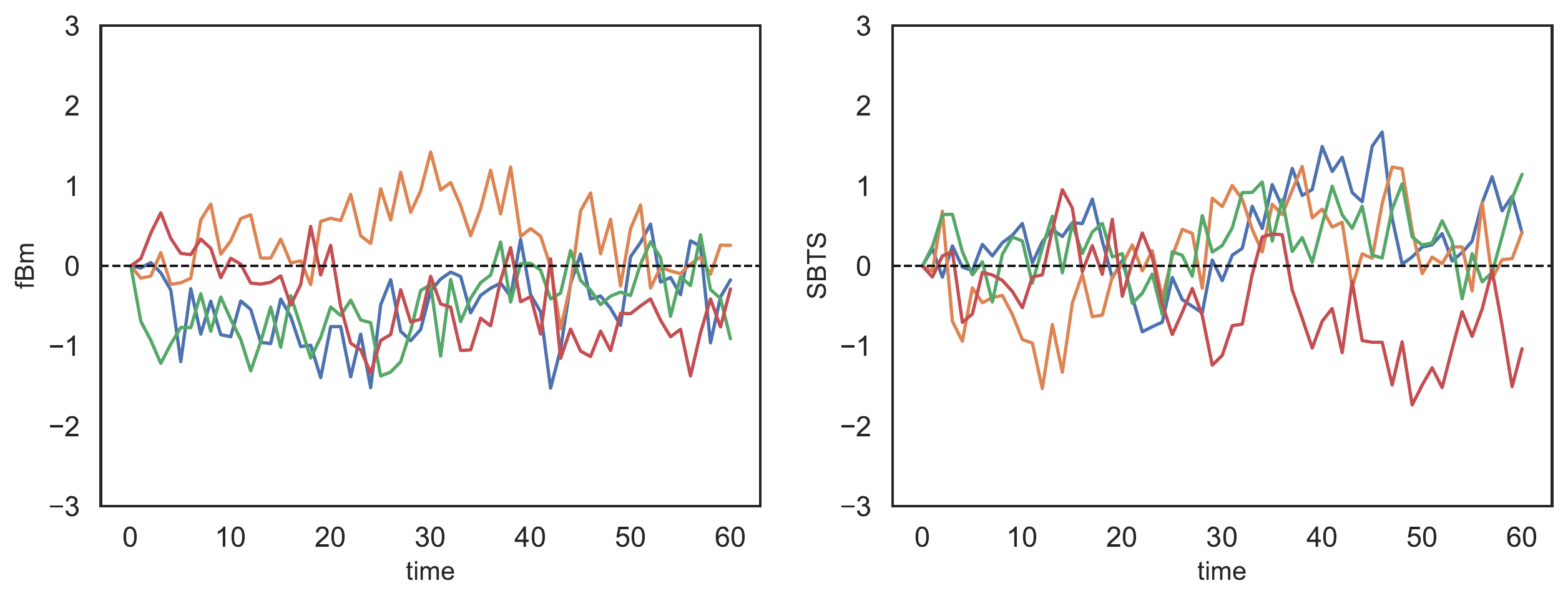}
    \caption{\footnotesize{Samples path of reference FBM  (left) and generator SBTS (right)}}
    \label{fig:FBM H=0.2}
\end{figure}

Figure \ref{fig:FBMcov} represents the covariance matrix of $(X_{t_1},\ldots,X_{t_N})$ for $N$ $=$ $60$ of the FBM and of the generated SBTS, while we plot in Figure \ref{fig:quadratic_var} the empirical distribution of the quadratic variation  $\sum_{i=0}^{N-1} |X_{t_{i+1}}-X_{t_i}|^2$ for the FBM and the SBTS.

\begin{figure}[H]
    \centering
    \includegraphics[width=13cm,height=6.5cm]{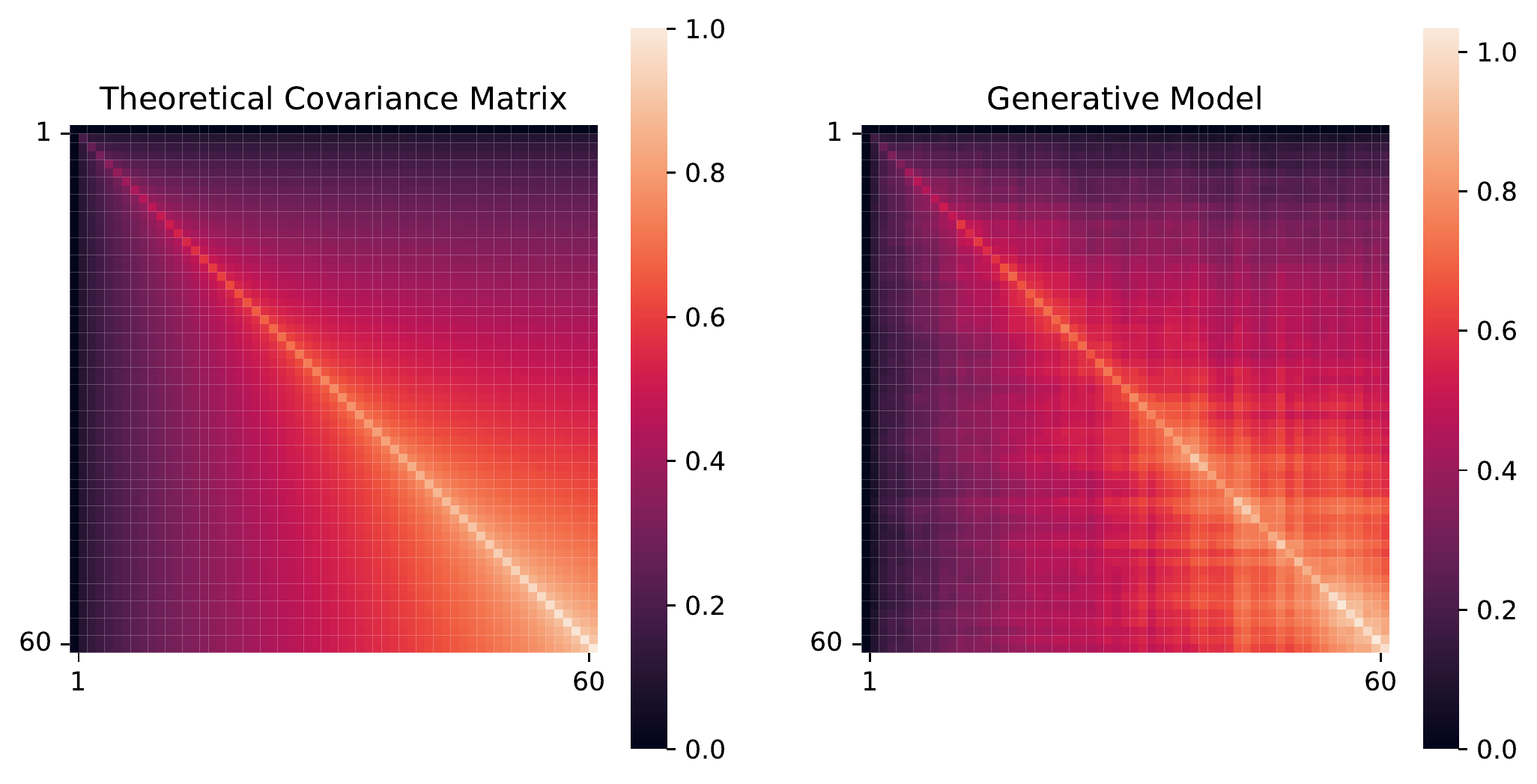}
    \caption{\footnotesize{Covariance matrix for reference FBM and SBTS}}
    \label{fig:FBMcov}
\end{figure}

\begin{figure}[H]
    \centering
    \includegraphics[width=13cm,height=6.5cm]{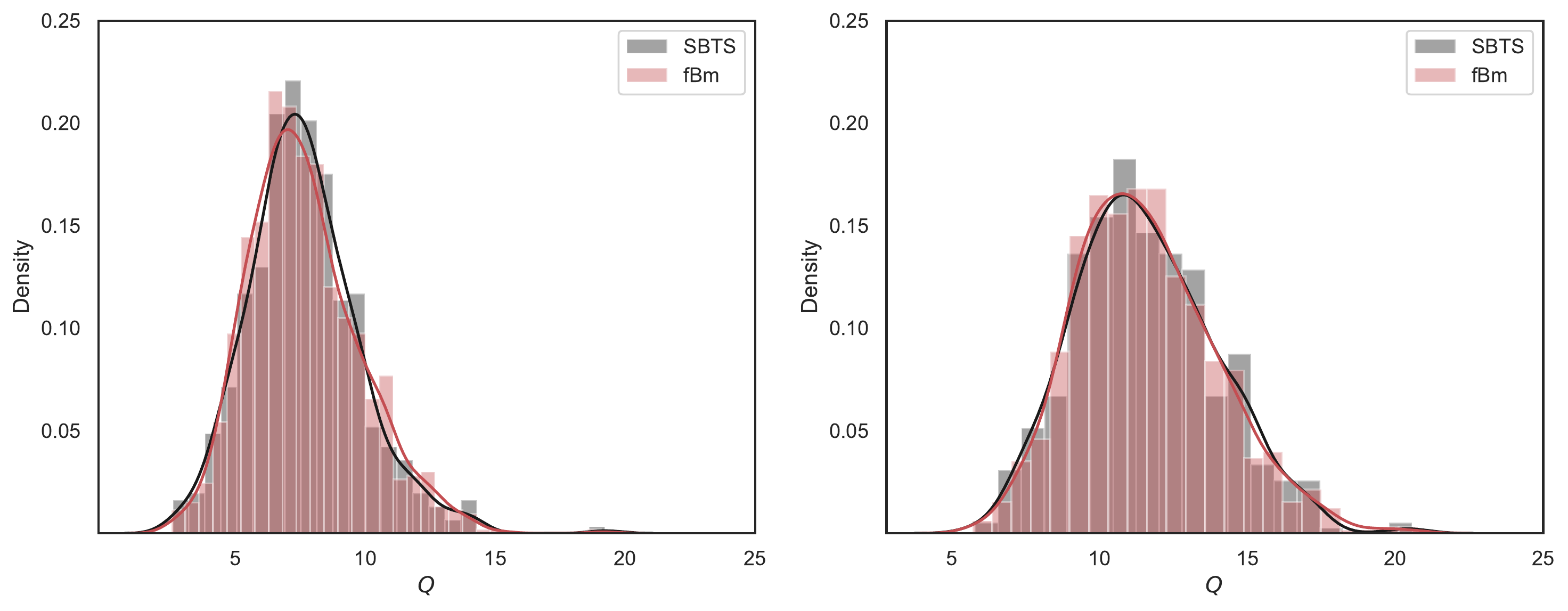}
    \caption{\footnotesize{Quadratic variation distribution for $N=30$ (left), $N=60$ (right) and $T=t_N=1$}}
    \label{fig:quadratic_var}
\end{figure}

Finally, we provide estimate of  the Hurst index from our generated SBTS with the standard estimator (see e.g. \cite{gaietal19}) given by: 
\begin{align}
 \hat H &= \; \frac{1}{2} \Bigg[ 1 - \frac{\log \Big( \Sum_{i=0}^{N-1} |X_{t_{i+1}} - X_{t_i}|^2\Big)}{\log N} \Bigg].  
\end{align}

For $N$ $=$ $60$, we get: 
$\hat H$ $=$ $0.2016$, Std $=$ $0.004$.

\subsection{Application to deep hedging on real-data sets}


In this paragraph, we use generated time series for applications to risk management, and notably the pricing of derivatives and the computation of associated hedging strategies via deep hedging approach. We use samples of historical
data for generating by SBTS new synthetic time series samples. We then compute deep hedging strategies that are trained from these synthetic samples, and we compare with the PnL and the replication error based on historical dataset. The general backtest procedure is illustrated in Figure \ref{fig:deepHedgingBacktest}.

\begin{figure}[H]
    \centering
    \includegraphics[width=13cm,height=5cm]{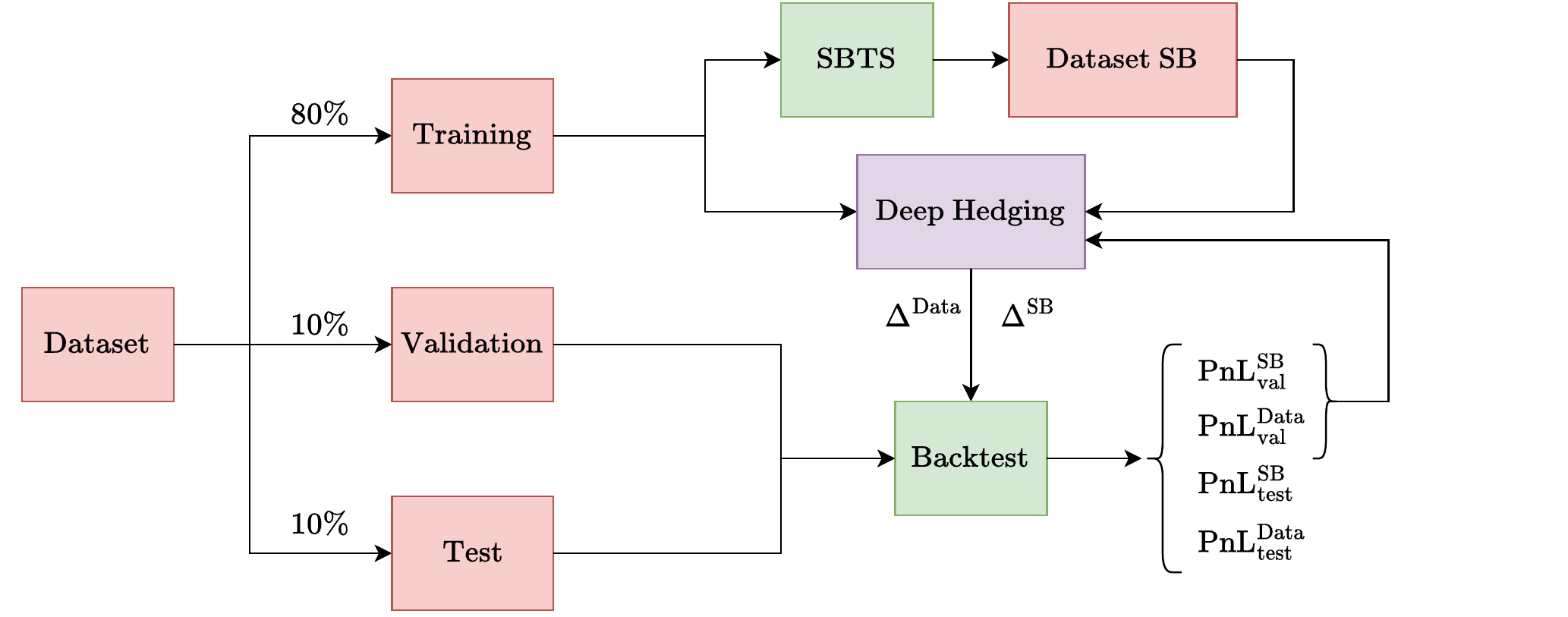}
    \caption{\footnotesize{Procedure of backtest for deep hedging}}
    \label{fig:deepHedgingBacktest}
\end{figure}


We consider stock price $S$ from the company {\it Apple} with  data (ticker is AAPL) from january 1, 2010 to january 30, 2020, and produce $M$ $=$ $2500$  samples of $N$ $=$  $60$ successive days, with a sliding window. The hyperparameters for the generation of SBTS synthetic samples are $N^\pi$ $=$ $100$, bandwith $h$ $=$ $0.05$.


We plot in Figure \ref{fig:AAPL_paths} four sample paths of the SBTS diffusion to be compared with the real ones from {\it Apple}. We illustrate the excess of kurtosis of the real data by plotting in Figure \ref{fig:fat-tailed-returns} 
the tail distribution for the return $R_{t_i}$ $=$ $\frac{S_{t_{i+1}}}{S_{t_i}}-1$: $x$ in $\log$-scale $\mapsto$ $\P[|R| \geq x]$, and found that the excess of kurtosis of real data is $1.96$, to be compared with the one from SBTS, and equal to $2.34$. 
Figure \ref{fig:AppleQV} represents the empirical distribution of the {\it Apple} time series data vs SBTS, while Figure  \ref{fig:Applecov} shows their covariance matrices.

\begin{figure}[H]
    \centering
    \includegraphics[width=12cm,height=4cm]{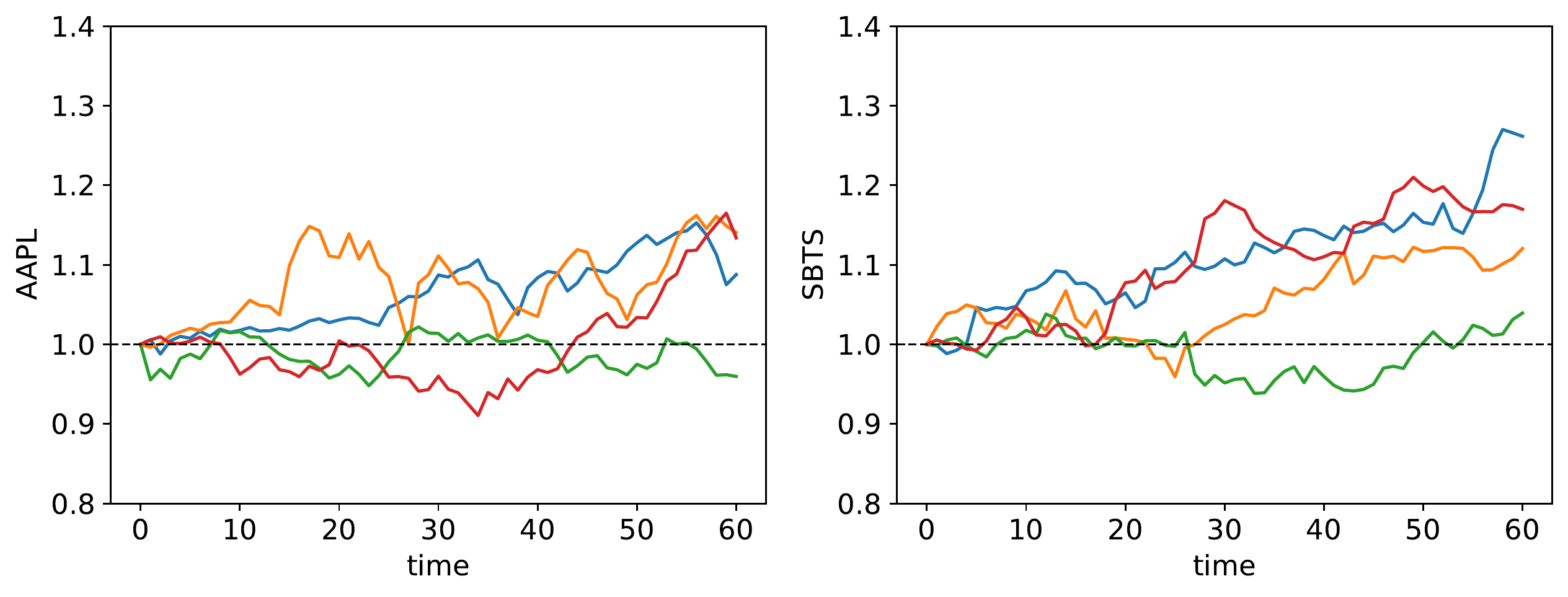}
    \caption{\footnotesize{Four paths  generated by Schrodinger bridge({\it Right}) vs real ones ({\it Left})}}
    \label{fig:AAPL_paths}
\end{figure}

\begin{figure}[H]
    \centering
    \includegraphics[width=7.7cm,height=5cm]{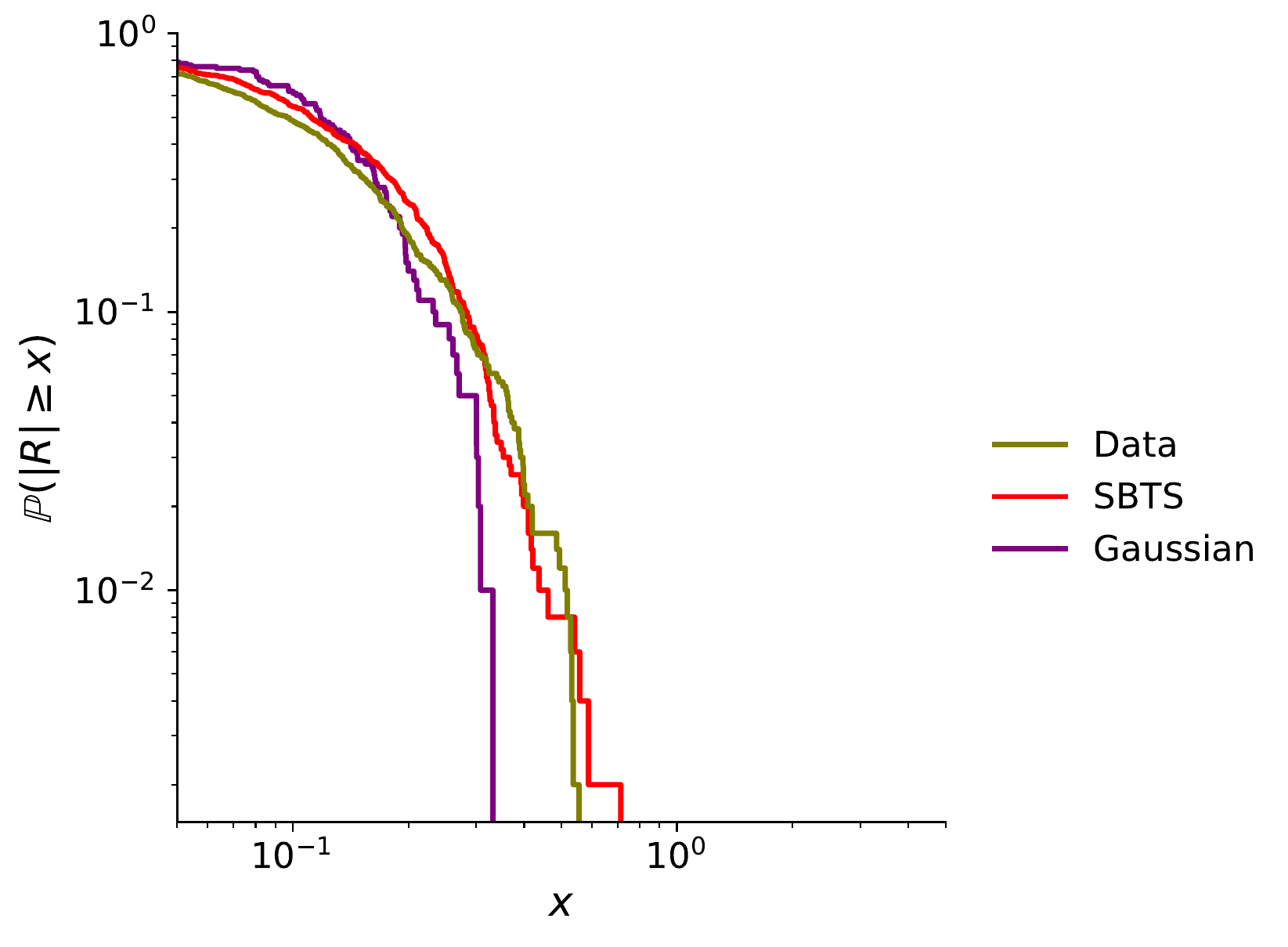}
    \caption{\footnotesize{Plot of  tail distribution for the return:  $x$ in $\log$-scale $\mapsto$ 
    $\P[|R| \geq x]$.  Excess of kurtosis for real-data = 1.96,  for generated SBTS =  2.34 }} 
    \label{fig:fat-tailed-returns}
\end{figure}

\begin{figure}[H]
    \centering
    \includegraphics[width=10cm,height=5cm]{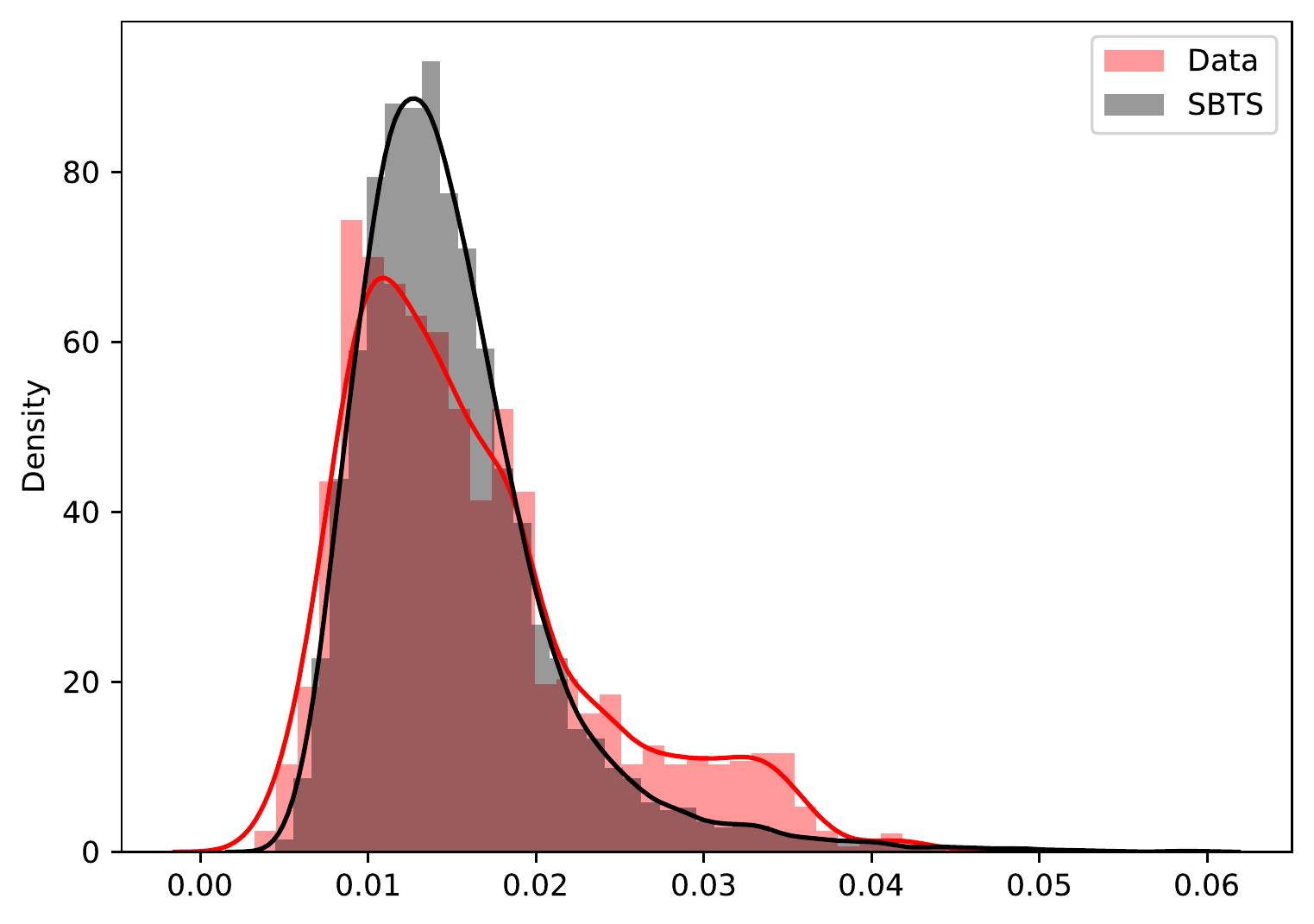}
    \caption{Comparison of quadratic variation distribution}
    \label{fig:AppleQV}
\end{figure}

\begin{figure}[H]
    \centering
    \includegraphics[width=13cm,height=5.5cm]{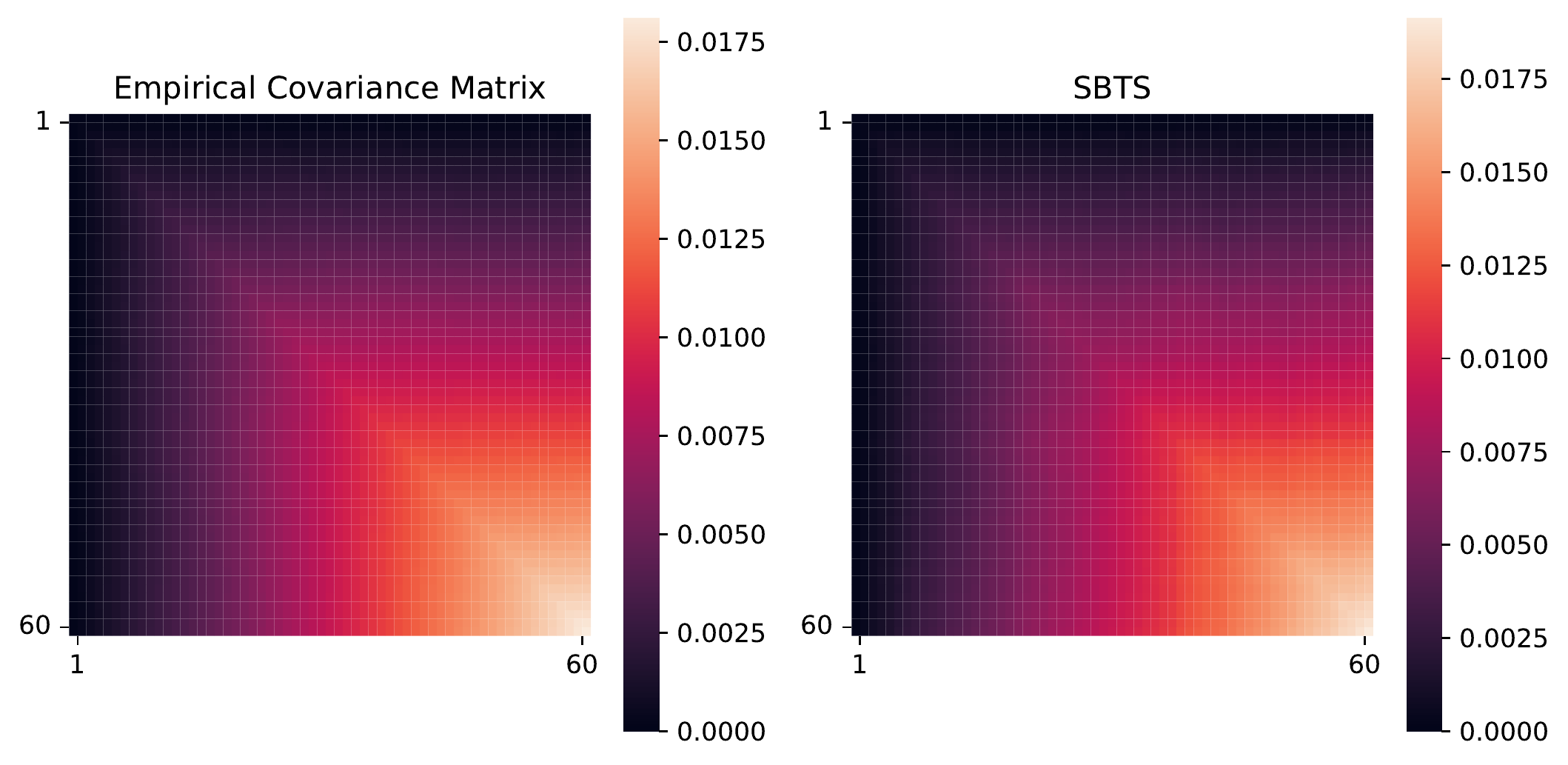}
    \caption{\footnotesize{Covariance matrix for real-data and generative SBTS}}
    \label{fig:Applecov}
\end{figure}

\vspace{2mm}

The synthetic time series generated by SBTS is now used for the deep hedging of ATM 
call option $g(S_T)$ $=$ $(S_T-S_0)_+$, i.e.,  by minimizing over the initial capital $p$ (premium)  and the parameters of the neural network $\Delta$ the (empirical)  loss function, called replication error: 
\begin{align}
 \E \big| {\rm PnL^{p,\Delta}} \big|^2, \quad \mbox{ with } \quad 
 {\rm PnL^{p,\Delta}} \; = \;  p + \sum_{i=0}^{N-1} \Delta(t_i,S_{t_i}) (S_{t_{i+1}} - S_{t_i}) \; - \; g(S_T). 
\end{align}
We then compare with the deep hedging on historical data by looking at the PnL and replication errors. The historical data set of Apple is split in  the chronological order, namely training data set from 01/01/2007 to 31/12/2017, validation data set from 01/01/2018 to 31/12/2028, and test set from 01/01/2019 to 30/01/2020. As pointed out in \cite{ruf}, it is important not to break the time structure as it may lead to an overestimation of the model performance.

In Figure \ref{fig:ApplePnL}, we plot the empirical distribution of the PnL with deep hedging obtained from real data vs SBTS, and backtested on the validation and test sets.  It appears that the PnL from SBTS has a smaller variance (hence smaller replication error), and yields less extreme values, i.e. outside the zero value, than the PnL from real data. This is also quantified in Table \ref{table:ApplePnL} where we note that the premium obtained from SBST is higher than the one from real data, which means that one is more conservative with SBTS by charging a higher premium.

\begin{figure}[H]
    \centering
    \includegraphics[width=6cm,height=5cm]{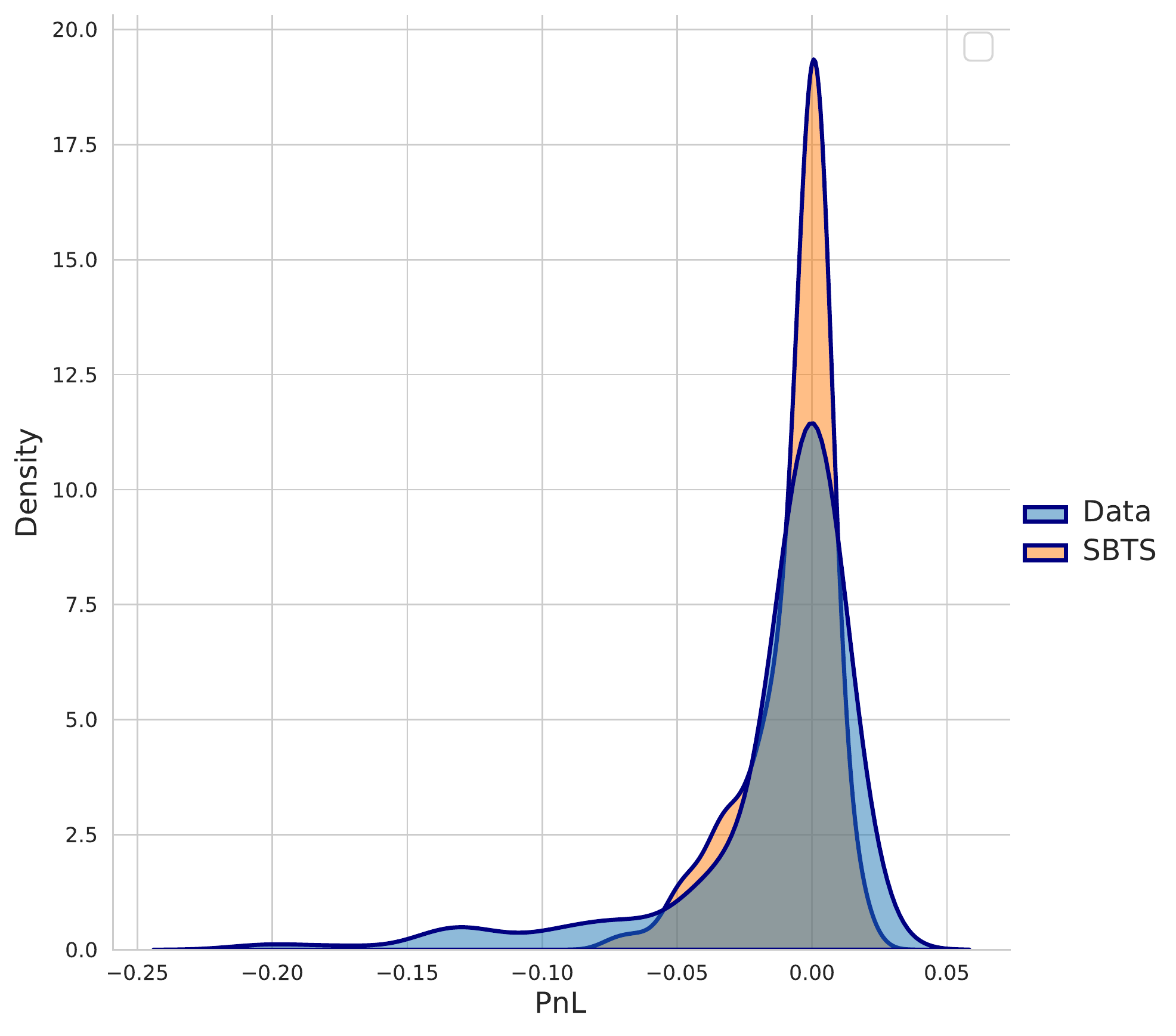}
    \includegraphics[width=6cm,height=5cm]{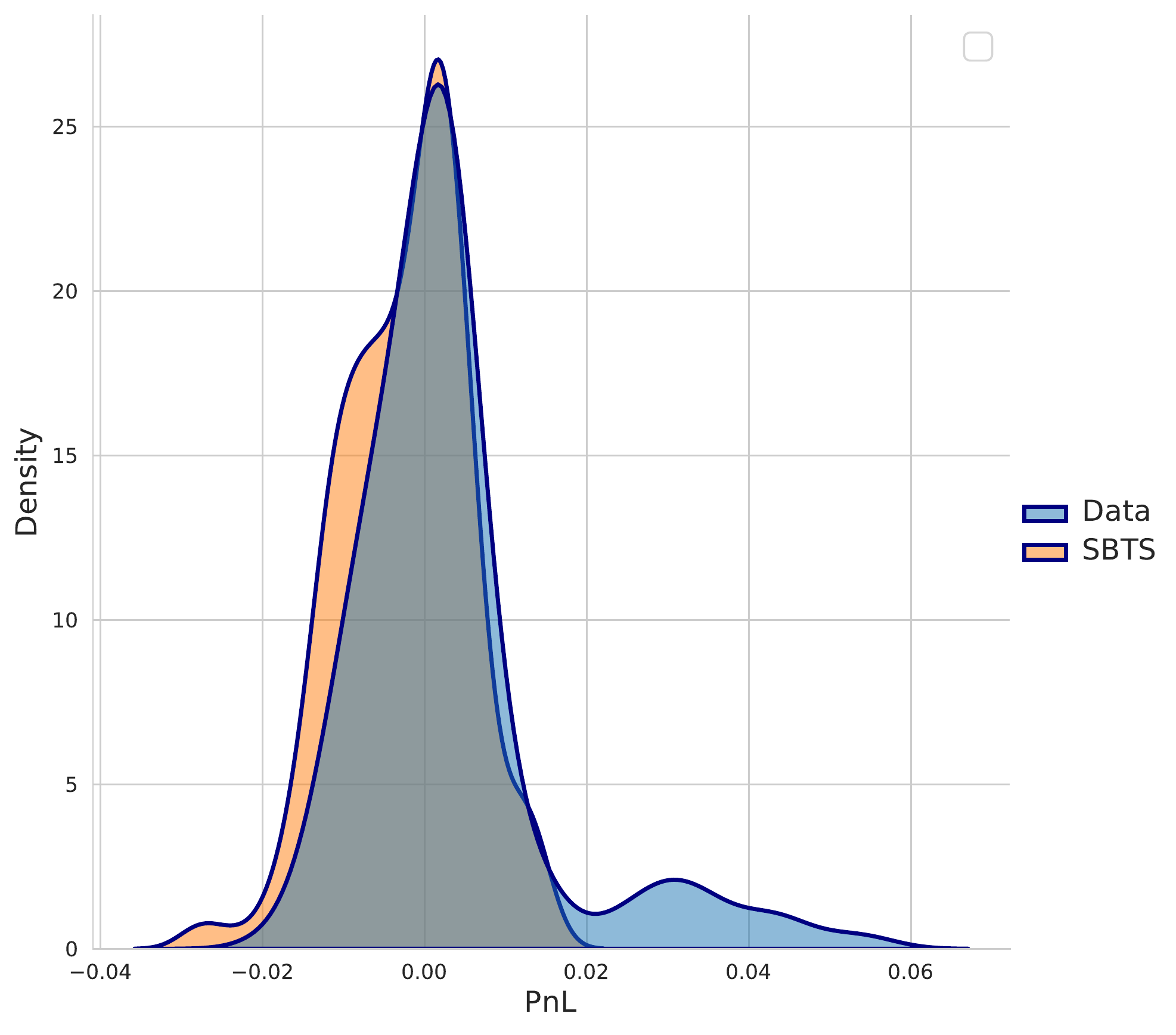}
    \caption{\footnotesize{Deep hedging PnL distribution with backtest from validation set ({\it left}) and test set ({\it right}).}}
    \label{fig:ApplePnL}
\end{figure}


\begin{table}[H]
\centering
\begin{tabular}{ll|ll|ll|ll|}
\cline{3-8}
                           &        & \multicolumn{2}{l|}{Training Set}    & \multicolumn{2}{l|}{Validation Set}   & \multicolumn{2}{l|}{Test Set}         \\ \cline{2-8} 
\multicolumn{1}{l|}{}      & Premium  & \multicolumn{1}{l|}{Mean}   & Std    & \multicolumn{1}{l|}{Mean}    & Std    & \multicolumn{1}{l|}{Mean}    & Std    \\ \hline
\multicolumn{1}{|l|}{Data} & 0.0415 & \multicolumn{1}{l|}{0.0008} & 0.0098 & \multicolumn{1}{l|}{-0.0154} & 0.0371 & \multicolumn{1}{l|}{0.003}   & 0.012  \\ \hline
\multicolumn{1}{|l|}{SBTS} & 0.0471 & \multicolumn{1}{l|}{0.0004} & 0.0109 & \multicolumn{1}{l|}{-0.0075} & 0.0164 & \multicolumn{1}{l|}{-0.0024} & 0.0076 \\ \hline
\end{tabular}
\caption{\footnotesize{Mean of PnL and its Std (replication error).}} \label{table:ApplePnL}
\end{table}

\section{Further tests in high dimension}

In this section, we illustrate how our SB approach can be used for generating samples in very high dimension. 


We use a data set of images from MNIST with training size $M$ $=$ $10000$. 
We first start with handwritten digital numbers, and plot in Figure \ref{fig:real_mnist_versus_SB} the static images from real MNIST data set and the ones generated by SBTS. The number of pixels is 
$28\times 28$, and the runtime for generating $16$ iamges is equal to $120$ seconds.

\begin{figure}[H]
    \centering
    \includegraphics[width=6cm,height=5.2cm]{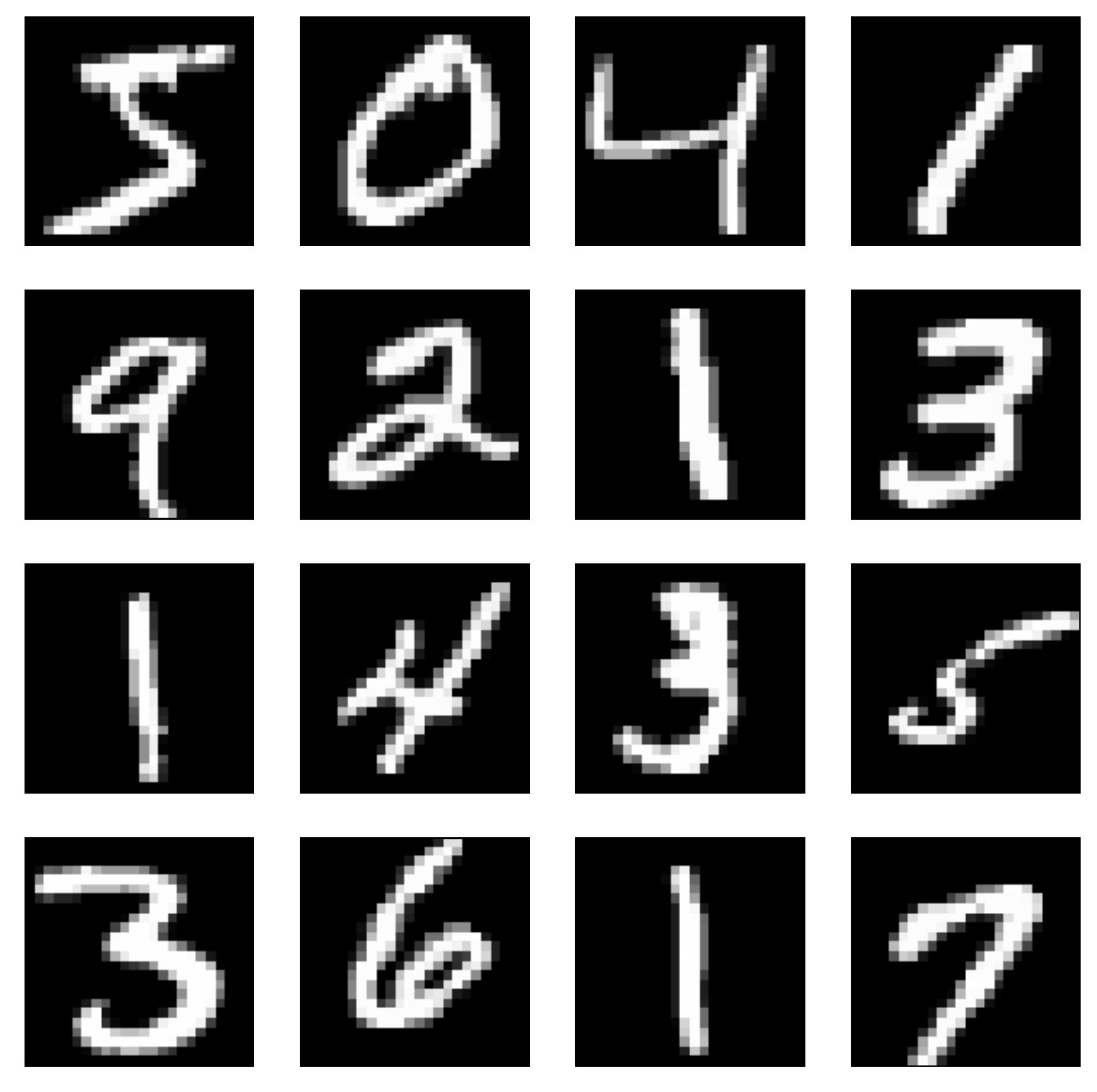}
    \quad
    \includegraphics[width=6cm,height=5.2cm]{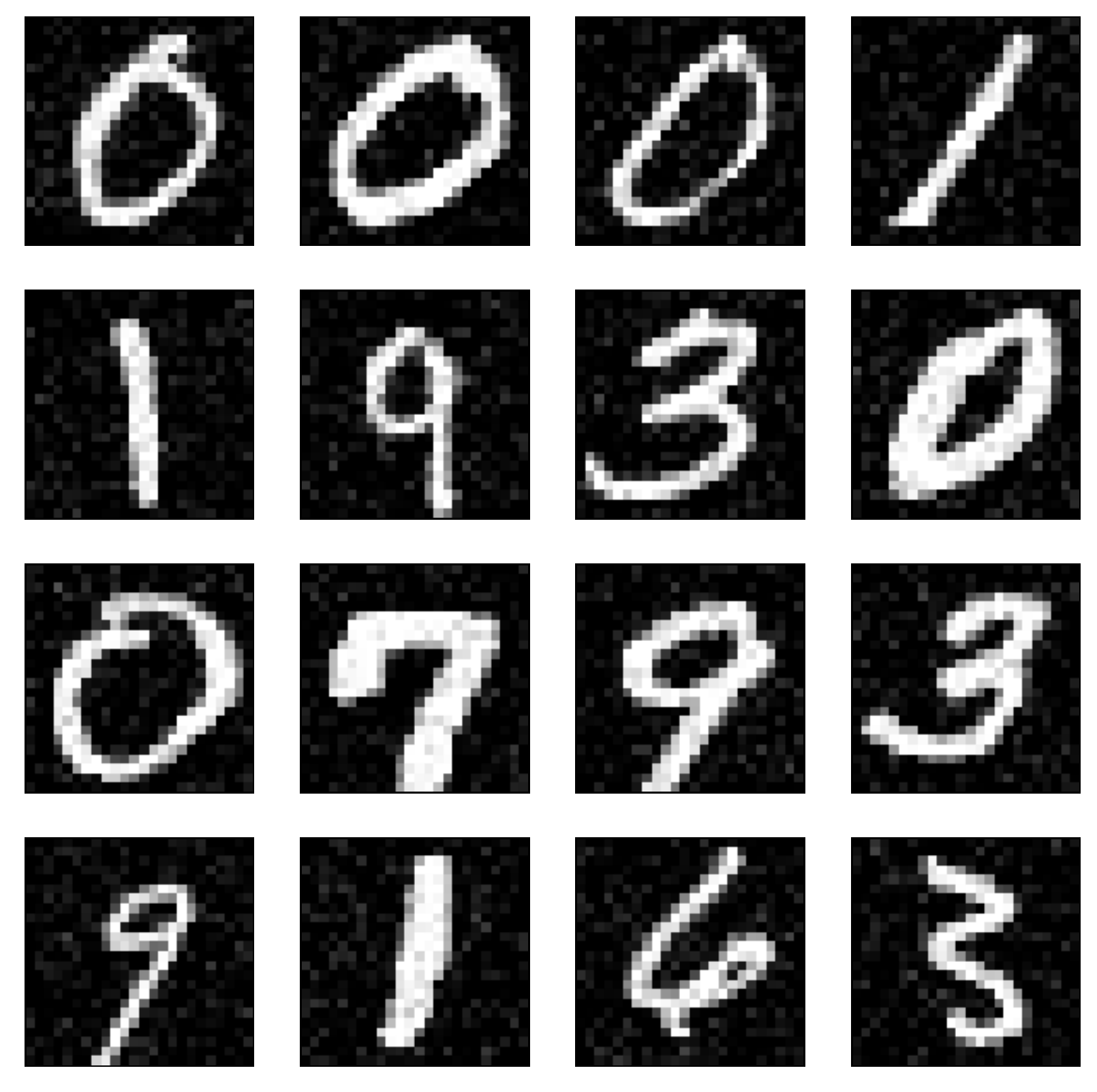}
    \caption{\footnotesize{{\it Left}: MNIST samples. {\it Right}: 
     Generated SBTS samples.}}
    \label{fig:real_mnist_versus_SB}
\end{figure}

Next, in Figures \ref{fig:real1_dist_mnist} and \ref{fig:real2_dist_mnist}, we plot sequential 
images sampled from real data set, and compare with the sequence generated by SBTS algorithm. 
The number of pixels is $14\times 14$, and the runtime for generating $100$ paths of sequential  
images is equal to $9$ minutes. 


\begin{figure}[H]
    \centering
    \includegraphics[width=12cm,height=1.9cm]{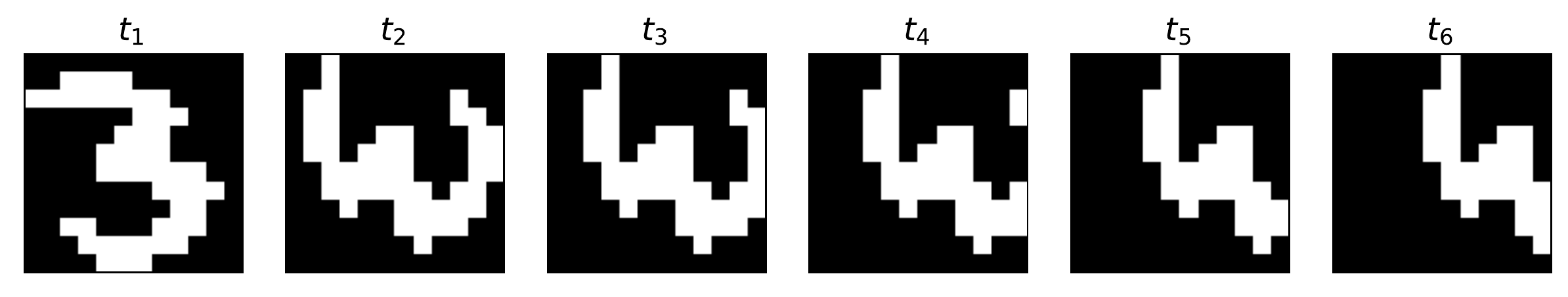}
    \\
    \includegraphics[width=12cm,height=1.9cm]{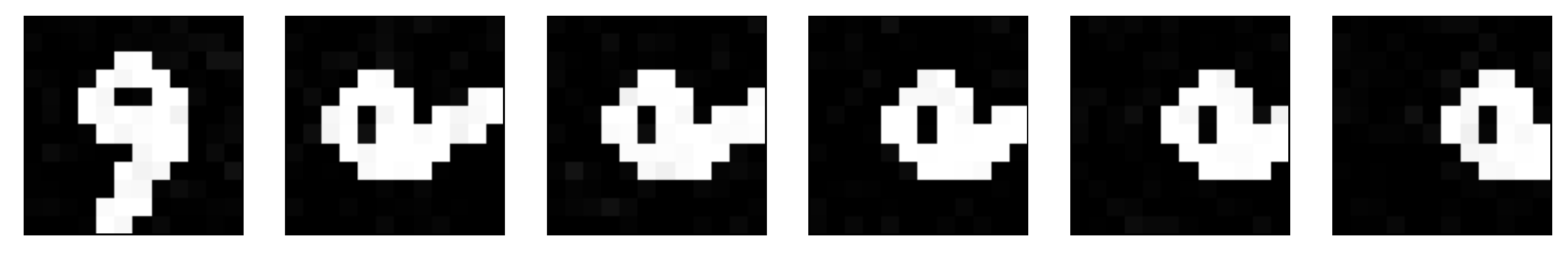}
    \caption{\footnotesize{{\it Top}: a time series sampled from real distribution. {\it Bottom}: generated time series via SBTS}}
    \label{fig:real1_dist_mnist}
\end{figure}

\begin{figure}[H]
    \centering
    \includegraphics[width=12cm,height=1.9cm]{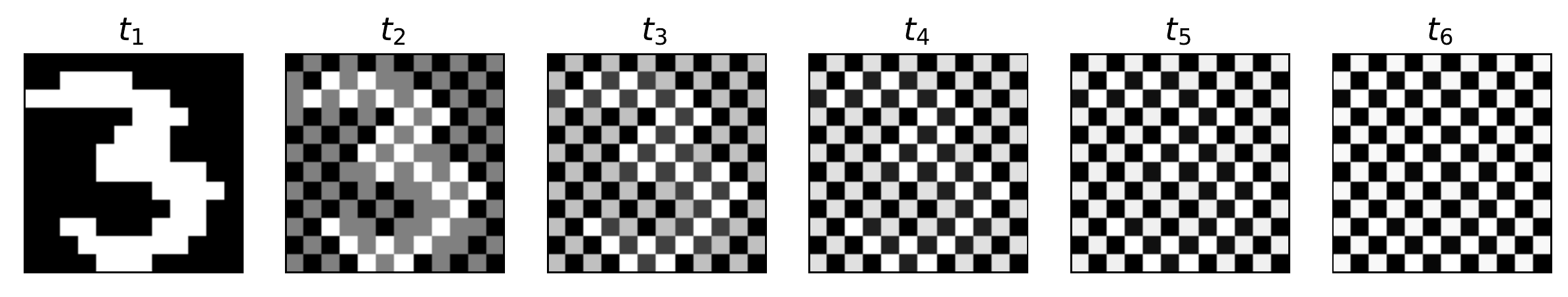}
    \\
    \includegraphics[width=12cm,height=1.9cm]{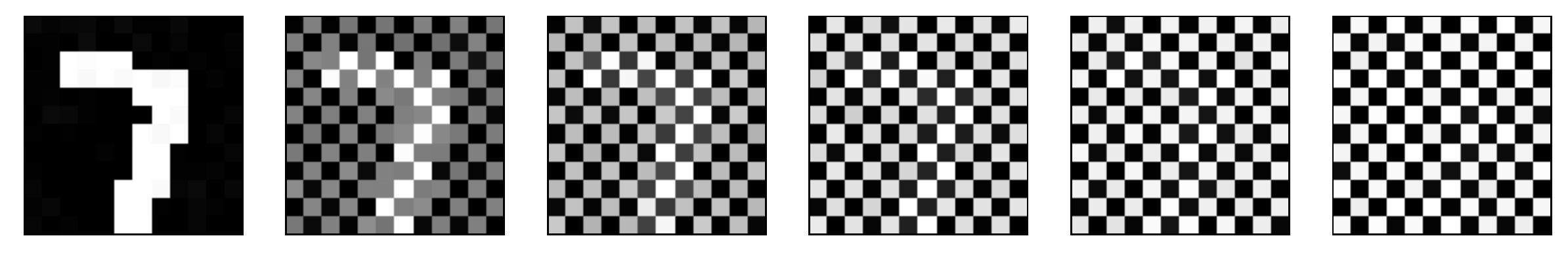}
    \caption{\footnotesize{{\it Top}: a time series sampled from real distribution. {\it Bottom}: generated time series via SBTS}}
    \label{fig:real2_dist_mnist}
\end{figure}

\bibliographystyle{plain}

\bibliography{biblioSBP}

\end{document}